\documentclass[reqno, 11pt]{amsart}
\usepackage{amsmath,amsxtra,amssymb,latexsym, amscd,amsthm}
\usepackage{marvosym}
\usepackage{eqname}
\usepackage{color}
\usepackage[numbers,sort&compress]{natbib}
\usepackage{booktabs}
\usepackage[utf8]{inputenc}
\usepackage[mathscr]{euscript}
\usepackage{mathrsfs}
\usepackage[english]{babel}
\usepackage{hyperref}
\usepackage{enumerate}
\usepackage{subfigure}
\usepackage{mathtools}

\newcommand{\R}{\Bbb{R}}
\newcommand{\N}{\Bbb{N}}
\newcommand{\vol}{\mathrm{vol}}
\newcommand{\inte}{\mathrm{int}}
\newcommand{\cl}{\mathrm{cl}}
\newcommand{\dist}{\mathrm{dist}}

\newcommand{\K}{\mathbf{K}}

\newcommand{\Sw}{\mathcal{S}_w}
\newcommand{\bx}{\boldsymbol x}
\newcommand{\by}{\boldsymbol y}
\newcommand{\be}{\boldsymbol\epsilon}
\newcommand{\bu}{\boldsymbol u}
\newcommand{\bb}{\mathbf{B}}

\newcommand{\tpsi}{\tilde{\psi}}

\newtheorem {theorem}{Theorem}[section]
\newtheorem {corollary}{Corollary}[section]

\newtheorem {proposition}{Proposition}[section]
\newtheorem {example}{Example}[section]
\newtheorem {definition}{Definition}[section]
\newtheorem {remark}{Remark}[section]
\newtheorem {assumption}{Assumption}[section]
\def\ar{a\kern-.370em\raise.16ex\hbox{\char95\kern-0.53ex\char'47}\kern.05em}
\def\ees{{\accent"5E e}\kern-.385em\raise.2ex\hbox{\char'23}\kern-.08em}
\def\eex{{\accent"5E e}\kern-.470em\raise.3ex\hbox{\char'176}}
\def\AR{A\kern-.46em\raise.80ex\hbox{\char95\kern-0.53ex\char'47}\kern.13em}
\def\EES{{\accent"5E E}\kern-.5em\raise.8ex\hbox{\char'23 }}
\def\EEX{{\accent"5E E}\kern-.60em\raise.9ex\hbox{\char'176}\kern.1em}
\def\ow{o\kern-.42em\raise.82ex\hbox{
  \vrule width .12em height .0ex depth .075ex \kern-0.16em \char'56}\kern-.07em}
\def\OW{O\kern-.460em\raise1.36ex\hbox{
\vrule width .13em height .0ex depth .075ex \kern-0.16em \char'56}\kern-.07em}
\def\UW{U\kern-.42em\raise1.36ex\hbox{
\vrule width .13em height .0ex depth .075ex \kern-0.16em \char'56}\kern-.07em}
\def\DD{D\kern-.7em\raise0.4ex\hbox{\char '55}\kern.33em}
\title{A new scheme for approximating the weakly efficient solution
set of vector rational optimization problems}
\author{Feng Guo}
\address[Feng Guo]{School of Mathematical Sciences, Dalian University of Technology, Dalian 116024, Liaoning Province, China.}
\email{fguo@dlut.edu.cn}

\author{Liguo Jiao$^{*}$}
\address[Liguo Jiao]{Academy for Advanced Interdisciplinary Studies, 
Northeast Normal University, Changchun 130024, Jilin Province, China.}
\email{hanchezi@163.com; jiaolg356@nenu.edu.cn}


\thanks{$^{*}$Corresponding Author}


\date{ \today}

\begin{document}
\maketitle

\begin{abstract}
In this paper, we provide a new scheme for approximating the weakly
efficient solution set for  
a class of vector optimization problems with rational objectives over
a feasible set defined by finitely many polynomial inequalities.
More precisely, we present a procedure to obtain a sequence of
explicit approximations of the weakly efficient solution set of the 
problem in question. Each
approximation is the intersection of the sublevel set of a single
polynomial and the feasible set. To this end, we make use of the
achievement function associated with the considered problem and
construct polynomial approximations of it over the feasible set from
above. Remarkably,
the construction can be converted to semidefinite programming
problems.
Several nontrivial examples are designed to illustrate the proposed new scheme.
\end{abstract}

\section{Introduction}\label{sec::1}

Vector optimization forms an important field of research in
optimization theory; see, e.g.,
\cite{Chankong1983,Ehrgott2005,Eschenauer1990,Luc2016,Sawaragi1985}, and many practical
applications in various areas, 
such as engineering \cite{Eschenauer1990}, humanitarian aid \cite{Gutjahr2016}, 
medical health \cite{Chen2012} and so on.
In this paper, we will be concerned with the following constrained
vector rational optimization problem of the form
\begin{align}\label{VP}
{\rm Min}_{\mathbb{R}^m_+} \; \left\{f(\bx):=\left(\frac{p_1(\bx)}{q_1(\bx)},\ldots, \frac{p_m(\bx)}{q_m(\bx)}\right): \bx \in \Omega\right\},  \tag{VROP}
\end{align}
where ``${\rm Min}_{\Bbb{R}^m_+}$" is understood with respect to the ordering non-negative 
orthant ${\Bbb{R}^m_+},$ $f: \mathbb{R}^n \rightarrow \mathbb{R}^m$ is a rational mapping 
with $f_i = \frac{p_i}{q_i},$ in which $p_i$ and $q_i$ are real polynomials in the variable $\bx = (x_1, \ldots, x_n)$ for each $i =1, \ldots, m,$ and the feasible set $\Omega$ is given by
\begin{align*}%\label{feasibleset}
\Omega:=\{ \bx \in \mathbb{R}^n \colon g_j(\bx) \geq 0,\ j = 1, \ldots, r\}, 
\end{align*}
where for each $j = 1, \ldots, r,$ $g_j$ is a real polynomial in the variable $\bx.$ 
By letting $q_i = 1$ for all $i = 1, \ldots, m$, our model then covers
vector polynomial optimization problems
\cite{Blanco2014,Kim2019,LeeJiao2021,Luc2016,Nie2021}, and by letting 
$p_i, q_i$ be linear functions for all $i = 1, \ldots, m$, our model also covers
linear fractional vector optimization problems \cite{Huong2020} as well.

For vector optimization, it is almost impossible to find a single
point simultaneously minimizing all the objective functions. 
Therefore, we usually look for some ``best preferred" solutions in
vector optimization. 
Now, let us recall the concepts of optimal solutions to vector
optimization problems.  A point $\bx \in \Omega$ is said to be 
an {\it efficient solution} (or Edgeworth–Pareto (EP) optimal point) to the
problem~\eqref{VP} if it holds that
\[ 
f(\by) - f(\bx) \not\in - \mathbb{R}^m_+\setminus\{0\}\quad
\textrm{for all} \quad \by \in \Omega;
\]
and a {\it weakly efficient solution} (or weakly EP optimal point) to
the problem~\eqref{VP} if it holds that 
\[
	f(\by) - f(\bx) \not\in -\mathbb{R}^m_{++} \quad \textrm{for all}
	\quad \by \in \Omega,
\]
where $\mathbb{R}^m_{++}$ denotes the positive orthant of
$\mathbb{R}^m$.
Let $\be\in\mathbb{R}_+^m$ be given, a point $\bx \in \Omega$ is said to be 
a {\it weakly $\be$-efficient solution} to the problem~\eqref{VP} if it holds that 
\[
	f(\by) - f(\bx) + \be \not\in -\mathbb{R}^m_{++} \quad \textrm{for
	all} \quad \by \in \Omega.
\]
Denote by $\mathcal{S}$ (resp., $\mathcal{S}_w$, $\Sw^{\be}$) the set of all efficient
(resp, weakly efficient, weakly $\be$-effcient) solutions to the problem~\eqref{VP}, respectively.
Clearly, $\mathcal{S} \subset \mathcal{S}_w \subset \Sw^{\be}$, but not conversely. 
We call the image $f(\Sw)$ the Pareto frontier (the Pareto curve if
$m=2$) of \eqref{VP}; see \cite{Magron2014}.

Throughout this paper, we make the following blanket assumptions on
\eqref{VP}:

\vskip 3pt
\begin{enumerate}[(\bf {A}1)]
	\item The feasible set $\Omega$ is nonempty and compact; 
	\item The denominators $q_i(\bx) >0$ over $\Omega$ for all $i = 1, \ldots, m.$ 
\end{enumerate}
As each $f_i$ is continuous, ({\bf A1}) implies that the image $f(\Omega)$ of the rational
mapping $f$ over $\Omega$ is also compact,
which ensures the existence of (weakly)
efficient solutions to the problem~\eqref{VP}; see, e.g.,
\cite[Theorem 1]{Borwein1983}, \cite[Theorem 2.1]{Ehrgott2005} and
\cite[Corollary 3.2.1]{Sawaragi1985}.
The problem~\eqref{VP} is well defined under ({\bf A2}), which is commonly adopted
	in the literature when studying fractional programming. Moreover,
	by substituting $\frac{p_iq_i}{q_i^2}$ for $\frac{p_i}{q_i}$, 
	{(\bf A2)} can be weakened as $q_i(\bx)\neq 0$
over $\Omega$ for all $i = 1, \ldots, m.$

Motivated by its extensive applications, a great deal of attention has
been attracted to the development of algorithms for computing (weakly)
efficient solutions  
to vector optimization; see
\cite{Burachik2017,Chuong2020,Fliege2009,LeeJiao2018,LeeJiao2021,Nie2021,Tanabe2019,
Tanabe2020,Tanabe2021} and references therein. 
Among them, there are mainly two different approaches for solving vector optimization, 
by which we mean finding its (weakly) efficient solutions. 
One is based on the scalarization methods (e.g., \cite{Burachik2017,Chuong2020,LeeJiao2018,LeeJiao2021,Nie2021}), 
which computes (weakly) efficient solutions by choosing some parameters
in advance and reformulating them as one or several single objective
optimization problems.
The other is based on descent methods;
see 
e.g., \cite{Fliege2009} for Newton's methods,
\cite{Tanabe2019,Tanabe2020,Tanabe2021,ZhaoYao2021,ZhaoYao2022} for (projected) gradient
methods.

We would like to emphasize that the aforementioned methods can only
find one or some particular (weakly) efficient solutions, rather than
giving information about the whole set of (weakly) efficient
solutions, which is apparently important for applications of vector
optimziation in the real world. 
Instead, the aim and novelty of this
paper is to provide a new scheme for approximating the whole set of
weakly ($\be$-)efficient solutions of \eqref{VP}.
More precisely, we
provide a procedure to obtain a sequence of explicit approximations of
$\Sw^{\be}$ (and hence $\Sw$ by letting $\be\to 0$). 
Each approximation is the intersection of the sublevel set of a
single polynomial and the feasible set $\Omega$.
As far as we know, there are few methods of this type for solving
vector optimization problems in the literature.

To this end, we make use of the {\it achievement function}
(c.f. \cite{Ehrgott2005,Pham2018,Wierzbicki1986})
associated with the problem~\eqref{VP} which is defined as
\[
\psi(\bx) := \sup_{\by \in {\Omega}} \min_{i = 1, \ldots, m} [f_i(\bx) - f_i(\by)].
\]
It can be shown that the sets $\Sw$ and $\Sw^{\be}$ can be written as
the intersection of sublevel sets of $\psi(\bx)$ and the feasible set
$\Omega$ (see Section \ref{sec::3}). As the function $\psi(\bx)$ can
be fairly complicated, the problem is reduced
to construct polynomial approximations of $\psi(\bx)$. By rewriting the
definition of $\psi(\bx)$ as a parametric polynomial optimization
problem, we can contruct a sequence of polynomial approximations
$\{\psi_k(\bx)\}_{k\in\mathbb{N}}$ of $\psi(\bx)$ over the feasible
set $\Omega$ from above by
invoking the ``joint+marginal'' approach developed by Lasserre in
\cite{Lasserre2009,Lasserre2015}. Remarkably,
the construction of $\{\psi_k(\bx)\}_{k\in\mathbb{N}}$ can be
converted to semidefinite programming (SDP) problems. For
$\be\in\mathbb{R}_+^m$ of the form $\be=(\delta,\ldots,\delta)$ with
$\delta>0$, the intersection, denoted by $\mathcal{A}(\delta, k)$, of
the sublevel set $\psi_k(\bx)\le\delta$ and the feasible set $\Omega$ are inner
approximations of $\Sw^{\be}$. Under some conditions, we prove that
$\vol\left(\Sw^{\be}\setminus\mathcal{A}(\delta,
k)\right)\to 0$ as $k\to\infty$, where ``$\vol(\cdot)$'' denotes the
Lebesgue volume (see Theorem \ref{th::main}). Since it holds 
for $\be=(\delta,\ldots,\delta)$ that
$\Sw^{\be}\to\Sw$ as $\delta\to
0$ (see Proposition \ref{prop::dist}), we may take $\mathcal{A}(\delta,
k)$ as an approximation of $\Sw$ with sufficiently small $\delta>0$
and sufficiently large $k\in\mathbb{N}$ (see Corollary
\ref{cor::convergence} and Remark \ref{rk::convergence}).

The rest of this paper is organized as follows.
Section~\ref{sec:preliminaries} contains some preliminaries on
polynomial optimization. In Section~\ref{sec::3}, we study the
characterization of the weakly efficient solution set of the problem
\eqref{VP} by the associated achievement function $\psi(\bx)$. In
Section~\ref{sec::4}, we show how to approximate the weakly
($\be$-)efficient solution set
of the problem~\eqref{VP}, and present some nontrivial illustrating
examples. Concusions are given in Section~\ref{sec::5}.

\section{Preliminaries}\label{sec:preliminaries}

In this section, we collect some notation and preliminary results
which will be used in this paper.
The symbol $\mathbb{N}$ (resp., $\mathbb{R}$, $\mathbb{R}_+$,
$\mathbb{R}_{++}$) denotes
the set of nonnegative integers (resp., real numbers, nonnegative real
numbers, positive real numbers).
For a set $D$ in $\mathbb{R}^n,$ we use ${\rm cl}(D)$ and ${\rm
int}(D)$ to denote the closure and interior of $D,$ respectively.
Denote by $\bb$ the closed unit ball in $\mathbb{R}^n$ centered at the
origin. For
a point $\bu\in\mathbb{R}^n$, $\dist(\bu, D)$ denotes the Euclidean
distance between $\bu$ and $D$. 
For $\bu\in \mathbb{R}^n$,
$\Vert \bu\Vert$ denotes the standard Euclidean norm of $\bu$.
For $\alpha:=(\alpha_1,\ldots,\alpha_n)\in\N^n$,
$|\alpha|=\alpha_1+\cdots+\alpha_n$.
For $k\in\N$, denote by $\N^n_k=\{\alpha\in\N^n \colon |\alpha|\le
k\}$ and $\vert\N^n_k\vert$ its cardinality.
Denote by $\mathbb{R}[\bx]$ the ring of polynomials in $\bx:=(x_1,
\ldots, x_n)$ with real coefficients and by $\mathbb{R}[\bx]_k$ 
the set of polynomials in $\mathbb{R}[\bx]$ of degree up to $k$.
For a polynomial $f,$ we use $\deg(f)$ to denote the total degree of $f.$
For $\alpha \in \mathbb{N}^n,$ the
notation $\bx^\alpha$ stands for the monomial $x_1^{\alpha_1}\cdots
x_n^{\alpha_n}.$ 

Now we recall some background about the {\it sum of squares}
representations of nonnegative (positive) polynomials over a set
defined by finitely many polynomial inequalities. We say that a
polynomial
$h\in\mathbb{R}[\bx]$ is sum of squares
of polynomials if there exist polynomials $h_j,$ $j = 1,\ldots,s,$
such that $h =\sum_{j=1}^{s}h_j^2.$
The set consisting of all sum of squares polynomial in $\bx$ is
denoted by $\Sigma^2[\bx].$
Let $\{h_1, \ldots, h_s\} \subset \mathbb{R}[\bx]$ be a finite set of
polynomials and 
\[
	S:=\{\bx\in\mathbb{R}^n \colon h_j(\bx)\ge 0,\ j=1,\ldots,s\}.
\]
\begin{assumption}\label{Archimedean}{\rm
		There exists some $N\in\mathbb{R}$ such that 
		\begin{eqnarray*}
	N - \sum_{i=1}^n x_i^2 =  \sigma_0(\bx)+\sum_{j = 1}^{s} \sigma_j(\bx) h_j(\bx),
\end{eqnarray*}
for some sum of squares polynomials $\sigma_j \in \Sigma^2[\bx], j=0,
1, \ldots, s.$
}\end{assumption}

\begin{theorem}[Putinar's
	Positivstellensatz {\upshape \cite{Putinar1993}}]\label{th::putinar}
	Suppose that Assumption \ref{Archimedean} holds. If
	$h(\bx)\in\mathbb{R}[\bx]$ is positive on $S,$ then $h(\bx)$ can
	be written in the form
	\begin{equation}\label{eq::pp}
h(\bx)=\sigma_0 (\bx) +\sum_{j = 1}^{s} \sigma_j h_j(\bx),
\end{equation}
for some sum of squares polynomials $\sigma_j \in \Sigma^2[\bx], j=0,
1, \ldots, s.$
\end{theorem}
Note that if we fix the degrees of $\sigma_j$'s in \eqref{eq::pp},
then checking the above representation of $h(\bx)$ reduces to an SDP
feasibility problem (c.f. \cite{Laurent2009}). The well-known
Lasserre's hierarchy of SDP
relaxations for polynomial optimization problems is based on Putinar's
Positivstellensatz and the dual moment theory (c.f.
\cite{Lasserre2001,Lasserre2009}).

A sparse version of the representation \eqref{eq::pp} is available if
some sparsity pattern is satisfied by $h$ and $h_j$'s. 
For a subset $I\subseteq\{1,\ldots,n\}$, denote the subset of
variables $\bx_I:=\{x_i\colon i\in I\}$ and $\mathbb{R}[\bx_I]$ as the
polynomial ring in the variables $\bx_I$.  
\begin{assumption}\label{assu::sparsity}
	{\rm
	There are partitions $\{1,\ldots,n\}=I_1\cup\cdots\cup I_l$ and
	$\{1,\ldots,s\}=J_1\cup\cdots\cup J_l$ where $J_i,$ $i = 1, \ldots, l$ are disjoint.	
	The collections $\{I_i\}_{i=1}^l$ and
	$\{J_i\}_{i=1}^l$ satisfy the following:
	\begin{enumerate}[\upshape (i)]
		\item $\forall i\in\{1,\ldots,l-1\}$,
			$\exists k\in\{1,\ldots,i\}$ s.t.
			$I_{i+1}\cap(I_1\cup\cdots\cup I_i)\subseteq I_k$;
		\item $h_j\in\mathbb{R}[\bx_{I_i}]$ for each $j\in J_i$, $1\le i\le l$.
		\item For each $i=1,\ldots,l$, there exists some $N_i\in\mathbb{R}$ such that 
\begin{eqnarray*}
	N_i - \sum_{j\in I_i} x_j^2 =  \sigma_{i,0}+\sum_{j \in J_i}
	\sigma_{i,j} h_j,
\end{eqnarray*}
for some sum of squares polynomials $\sigma_{i,0}, \sigma_{i,j} \in
\Sigma^2[\bx_{I_i}],$ $j\in J_i.$
	\end{enumerate}}
\end{assumption}
The following result enables us
to construct sparse SDP relaxations of polynomial optimization
problems, which can significantly reduce the computational cost (c.f.
\cite{LasserreSparsity,WKKMsparsity}). 

\begin{theorem}[Sparse version of Putinar's
	Positivstellensatz {\upshape
	\cite{LasserreSparsity,WKKMsparsity,GNS2007}}]\label{th::sputinar}
	Suppose that Assumption \ref{assu::sparsity}
	holds. If $h(\bx)\in\sum_{i=1}^l\mathbb{R}[\bx_{I_i}]$ and is positive on
	$S,$ then $h(\bx)$ can be written as
\[
	h(\bx)=\sum_{i=1}^l\left(\sigma_{i,0} + \sum_{j\in J_i}
	\sigma_{i,j} h_j\right),
\]
for some sum of squares polynomials $\sigma_{i,0},
\sigma_{i,j}\in\Sigma^2[\bx_{I_i}]$, $j\in J_i$, $i=1,\ldots,l$.
\end{theorem}

\section{Charactering the weakly efficient solution set}\label{sec::3}
In this section, we study the achievement function associated with
\eqref{VP}, which can be used to characterize the weakly
($\be$-)efficient solution set of \eqref{VP}. 

By defintion of $\Sw$, we have
\begin{align*}
\mathcal{S}_w =&\ \left\{\bx \in \Omega \colon \forall \by \in \Omega,\
	f(\by) - f(\bx) \not\in -\mathbb{R}^m_{++}\right\} \\
=&\ \left\{\bx \in \Omega \colon \forall \by \in \Omega,\ \exists i \in
	\{1, \ldots, m\}\ \textrm{such that}\ f_i(\bx) - f_i(\by) \leq 0\right\} \\
=&\  \left\{\bx \in \Omega \colon \forall \by \in \Omega,\ \min_{i = 1,
\ldots, m} [f_i(\bx) - f_i(\by)] \leq 0 \right\} \\
=&\ \left\{\bx \in \Omega \colon \sup_{\by \in \Omega} \min_{i = 1,
\ldots, m} [f_i(\bx) - f_i(\by)] \leq 0 \right\}.
\end{align*}
Let $\psi: \mathbb{R}^n \rightarrow \mathbb{R}$ be the function given by
\begin{align*}%\label{achievement_func}
\psi(\bx) := \sup_{\by \in {\Omega}} \min_{i = 1, \ldots, m} [f_i(\bx) - f_i(\by)].
\end{align*}
The function $\psi(\bx)$ is known as the achievement function in the
area of vector optimization in the literature; see \cite[Section
4.6]{Ehrgott2005} and \cite{Pham2018,Wierzbicki1986}.
Therefore,
\begin{align*}
\mathcal{S}_w = \left\{\bx \in \R^n \colon \psi(\bx) \leq 0\right\} \cap \Omega.
\end{align*}
Moreover, we have the following results, which imply that the function
$\psi(\bx)$ is indeed a {\it merit function} (see \cite{Dutta2017,LiuNgYang2009,Tanabe2020,Tanabe2021}).

\begin{proposition}{\upshape \cite[Lemmas 3.1 and
	3.2]{Pham2018}}\label{jiao12.28a}
	The achievement function $\psi(\bx)$ satisfies
\begin{itemize}
\item[{\rm (i)}] $\psi(\bx) \geq 0$ for all
	$\bx \in \Omega$ and hence $\mathcal{S}_w = \left\{\bx \in \Omega
	\colon \psi(\bx)=0\right\}.$ 
\item[{\rm (ii)}] $\psi(\bx)$ is locally Lipschitz on $\Omega.$
\end{itemize} 
\end{proposition}
\begin{proof}
(i)	is clear. If the objective in \eqref{VP} is a vector of
polynomials, (ii) was proved in \cite[Lemma 3.2]{Pham2018} which is
based on the locally Lipschitz property of polynomial functions. Note
that the rational function $f_i$ is locally Lipschitz on $\Omega$ under ({\bf
A1-2}). Hence, the proof of \cite[Lemma 3.2]{Pham2018} is still valid for the
case studied in this paper.
\end{proof}
So far, we know 
%under {(\bf A1)} (which guarantees the existence of weakly efficient solutions to the problem~\eqref{VP}),
the weakly efficient solution set $\mathcal{S}_w$ can be {\it
completely} characterized with the help of the achievement function
$\psi(\bx)$. 
Note that, $\psi(\bx)$ can be fairly complicated and computing
$\psi(\bx)$ by some descent methods directly might be difficult. 
%since very few information is known for the function $\psi$.
However, as shown below in Proposition~\ref{jiao10.12a}, the sublevels
of $\psi(\bx)$ 
have rather close relation with the set of all weakly $\be$-efficient solutions,
which in turn yields the information of the set of all weakly efficient solutions. 

Recall the definition of the set $\Sw^{\be}$ of all weakly $\be$-efficient
solutions to \eqref{VP}, and clearly by definition, 
$\mathcal{S}_w \subset \mathcal{S}^{\be}_w$ for any $\be\in\mathbb{R}^m_+$. 
Conversely, denote a set-valued mapping
$\mathcal{F}(\cdot): \R^m \rightrightarrows \R^n$ and let 
$\mathcal{F}(\be):=\Sw^{\be}$ for $\be\in\mathbb{R}^m_+$. The
following proposition shows that
$\mathcal{F}(\cdot)$ is {\it continuous at $\bar\be = 0$ relative to $\mathbb{R}^m_+$} 
in the sense of Painlev\'e--Kuratowski
(see \cite[Definition 5.4]{RockafellarWets98}), i.e.,
$\mathcal{F}({\be})\to\mathcal{F}(0)$ as $\be\to 0$.
For convenience, we recall the definitions of continuity
(outer semicontinuity, inner semicontinuity) for set-valued mapppings; 
see \cite[Chapters 4 \& 5]{RockafellarWets98} for more information.
Given a set-valued mapping $F: \R^m \rightrightarrows \R^n,$ 
we denote by 
\[
\limsup_{\by \to \bar\by} F(\by):=\left\{\bx \in \R^n \colon \exists 
\by_k \to \bar \by,\ \exists \bx_k \to \bx \ \textrm{with}\ \bx_k \in F(\by_k)\right\},
\]
\[
\liminf_{\by \to \bar\by} F(\by):=\left\{\bx \in \R^n \colon \forall 
\by_k \to \bar \by,\ \exists \bx_k \to \bx \ \textrm{with}\ \bx_k \in F(\by_k)\right\},
\]
the outer and inner limit of $F$ at $\bar\by$ in the sense of 
Painlev\'e--Kuratowski, respectively.
\begin{definition}\label{def520}{\rm
A set-valued mapping $F: \R^m \rightrightarrows \R^n$ is said to be {\it 
outer semicontinuous} (osc) at $\bar \by$ if 
$
\limsup\limits_{\by \to \bar\by} F(\by) \subset F(\bar\by),
$
and {\it inner semicontinuous} (isc) at $\bar\by$ if 
$
F(\bar\by) \subset \liminf\limits_{\by \to \bar\by} F(\by).
$
It is called {\it continuous} at $\bar\by$ if $F$ is 
simultaneously osc and isc at $\bar\by,$ i.e., $F({\by})\to F(\bar\by)$ as $\by\to \bar\by$.
These terms are invoked {\it relative to $X,$} a subset of $\R^m$ containing 
$\bar\by,$ if the inclusions hold in restriction to convergence $\by \to
\bar\by$ with $\by \in X.$ 
}\end{definition}

It follows from Definition~\ref{def520} that
$\mathcal{F}(\cdot)$ is continuous at $\bar\be = 0$ relative to $\mathbb{R}^m_+$.
Similar to \cite[Proposition 5.12 and Exercise 5.13]{RockafellarWets98}, we
have the following result.
For any $\be \in \R^m_+$, denote $\be_{\max}:=\max\limits_{i
= 1, \ldots, m}\{ \epsilon_i\}$ and $\be_{\min}:=\min\limits_{i = 1, \ldots, m}\{ \epsilon_i\}.$ 
\begin{proposition}\label{prop::dist}
For any $d>0,$ there exists a number $\delta(d)>0$ depending on $d$
such that $\dist (\bu, \Sw) < d$ for any $\bu\in\Sw^{\be},$ i.e.$,$
$\Sw^{\be}\subset\Sw+d\bb,$ whenever $\be_{\max}<\delta(d)$. 
\end{proposition}
\begin{proof}
	Suppose that the conclusion does not hold for 
	some $d>0$. Then, for any $k\in\mathbb{N}$, there exist
	$\be^{(k)}$ with $\be^{(k)}_{\max}<\frac{1}{k}$ and a
	point $\bu^{(k)}\in\Sw^{\be^{(k)}}$ such that $\dist(\bu^{(k)},
	\Sw)\ge d$. As $\Omega$ is compact, without loss of generality, we
	can assume that there is a point $\bu'\in\Omega$ such that
	$\lim_{k\rightarrow\infty} \bu^{(k)}=\bu'$. 
    Now we show that
	$\bu'\in\Sw$. To the contrary, suppose that there exists
	$\by'\in\Omega$ such that $f(\by')-f(\bu')\in - \mathbb{R}^m_{++}$,
	i.e., $\max_{i = 1, \ldots, m} [f_i(\by')-f_i(\bu')] < 0$. Due to the
	continuity of $f_i$, there exists $k'\in\mathbb{N}$ such that for
	each $i = 1, \ldots, m$, 
	\[
		\max_{i = 1, \ldots, m} [f_i(\by')-f_i(\bu')] + \frac{1}{k} +
		f_i(\bu') -
		f_i(\bu^{(k)})<0
	\]
	holds for any $k\ge k'$. Then for each $i = 1, \ldots, m$,
	\begin{align*}
			f_i(\by')-f_i(\bu^{(k)})+\be^{(k)}_i&=f_i(\by')-f_i(\bu')+f_i(\bu')-f_i(\bu^{(k)})+\be^{(k)}_i\\
			&\le \max_{i = 1, \ldots, m} [f_i(\by')-f_i(\bu')] +
			f_i(\bu')-f_i(\bu^{(k)})+ \frac{1}{k} \tag{by $\be^{(k)}_{\max}<\frac{1}{k}$} \\
          &<0,
		\end{align*}
	which means that $f(\by')-f(\bu^{(k)}) + \be^{(k)} \in - \mathbb{R}^m_{++}$,
	i.e., $\bu^{(k)}\not\in\Sw^{\be^{(k)}}$, a contradiction.
	Hence, $\bu'\in\Sw$ and $\dist(\bu', \Sw)=0$. However, due to the
	continuity of distance function, one has
	\[
		\dist(\bu', \Sw)=\lim_{k\rightarrow\infty} \dist(\bu^{(k)},
		\Sw)\ge d>0,
	\]
	a contradiction. 
\end{proof}

Furthermore, the following proposition allows us to study the
set $\Sw^{\be}$ of all weakly $\be$-efficient solutions by means of
sublevels of $\psi(\bx)$.

\begin{proposition}\label{jiao10.12a}
	For any $\be \in \R^m_+,$ we have
\begin{align}\label{jiao10.08b}
	\left\{\bx \in \Omega \colon \psi(\bx) \leq \be_{\min} \right\}
	\subset
	\mathcal{S}^{\be}_w\subset
	\left\{\bx \in \Omega \colon \psi(\bx) \leq \be_{\max} \right\}.
\end{align}
	Particularly$,$ if $\be_{\max}=\be_{\min},$ then 
	\[
	\left\{\bx \in \Omega \colon \psi(\bx) \leq
		\be_{\min}=\be_{\max} \right\}=\Sw^{\be}. 
	\]
\end{proposition}
\begin{proof}
To show the first relation in \eqref{jiao10.08b}, suppose to the contrary that there exists $\bu\in\Omega$ 
such that $\psi(\bu)\le\be_{\min}$ but $\bu\not\in\Sw^{\be}$. Then, there
exists $\by'\in\Omega$ such that $f(\by')-f(\bu)+\be\in -
\mathbb{R}^m_{++}$, i.e., $f_i(\bu)-f_i(\by')-\be_i > 0$ for each
$i = 1, \ldots, m$. Thus, $\min_{i=1,\ldots,m}
[f_i(\bu)-f_i(\by')]>\be_{\min}$ which implies that
$\psi(\bu)>\be_{\min}$, a contradiction.

Now, fix a point $\bu\in\Sw^{\be}$. For any $\by\in\Omega$, by
definition, there exists $k_{\by}\in\{1,\ldots,m\}$ depending on $\by$ such
that $f_{k_{\by}}(\by)-f_{k_{\by}}(\bu)+\be_{k_{\by}}\ge 0$. Then,
$\min_{i=1,\ldots,m} [f_i(\bu)-f_i(\by)]\le \be_{k_{\by}}$ for all $\by\in\Omega,$
and hence
\[
	\psi(\bu)=\max_{\by\in\Omega}\min_{i=1,\ldots,m}[f_i(\bu)-f_i(\by)]\le
	\be_{\max},
\]
thus, the second relation in \eqref{jiao10.08b} holds. 
Consequently, the conclusion follows.
\end{proof}

\section{Approximations of weakly ($\be$-)efficient solution set}\label{sec::4}
In this section, we will construct polynomial approximations of the
achievement function $\psi(\bx)$ from above and use their sublevel sets to
approximate the set of all weakly ($\be$-)efficient solutions
to \eqref{VP}. The construction of these polynomial approximations of
$\psi(\bx)$ is inspired by \cite{Lasserre2015} and can be reduced to
SDP problems.  As $\Omega$ is compact, after a possible re-scaling of
the $g_j$'s, we may and will assume that $\Delta:=[-1,
1]^n\supseteq\Omega$ in the rest of this paper.

\subsection{Approximations of achievement function}
To construct polynomial approximations of $\psi(\bx)$, we need first
compute upper and lower bounds of $f_i(\bx)$, $i=1,\ldots,m,$ over $\Omega$. 
To this end, for each $i=1,\ldots, m$, we compute a number
$f_i^{\text{lower}}\in\mathbb{R}$ satisfying
\begin{equation}\label{eq::f_*}
	\begin{aligned}
		&p_i(\bx)-f_i^{\text{lower}} q_i(\bx)
		=\sigma_{i,0}(\bx)+\sum_{j=1}^r\sigma_{i,j}(\bx)g_j(\bx)+\sum_{j=1}^n\sigma_{i,r+j}(\bx)(1-x_j^2),\\ 
		&\sigma_{i,0}, \sigma_{i,j}\in\Sigma^2[\bx], \
		j=1,\ldots,r+n,\  \deg(\sigma_{i,0})\le 2k_i, \
		k_i\in\mathbb{N},\\
		&\deg(\sigma_{i,j}g_j)\le 2k_i,\ j=1,\ldots,r, \  
		\deg(\sigma_{i,r+j}(1-x_j^2))\le 2k_i,\ j=1,\ldots,n, 
	\end{aligned}
\end{equation}
which is equivalent to an SDP feasibility problem (c.f. \cite{Laurent2009}). 
Under ({\bf A1-2}), each $\frac{p_i(\bx)}{q_i(\bx)}$ is bounded from below on
$\Omega$ and $p_i(\bx)-f_i^{\text{lower}} q_i(\bx)>0$ on $\Omega$ for
any $f_i^{\text{lower}}<\min_{\bx\in\Omega}\frac{p_i(\bx)}{q_i(\bx)}$.
Hence, by Putinar's Positivstellensatz, a number $f_i^{\text{lower}}$
satisfying \eqref{eq::f_*} always exists for $k_i$ large enough (note
that Assumption \ref{Archimedean} holds due to the redundant
polynomials $1-x_j^2$, $j=1,\ldots,n,$ added in \eqref{eq::f_*}).
Clearly, it holds that  
\[
	f^{\text{lower}}:=\min_{i=1,\ldots,m}f_i^{\text{lower}}\le \min_{i
	= 1, \ldots, m,\ \bx \in \Omega}\frac{p_i(\bx)}{q_i(\bx)}. 
\]
Similarly, replace $p_i(\bx)-f_i^{\text{lower}} q_i(\bx)$ in
\eqref{eq::f_*} by $f_i^{\text{upper}} q_i(\bx) - p_i(\bx)$,
where $f_i^{\text{upper}}$ denotes another real number. Then
similarly, such a number $f_i^{\text{upper}}$ exists for $k_i$ large
enough and can be computed by solving another SDP feasibility problem.
Then, we have 
\[
	f^{\text{upper}}:=\max_{i=1,\ldots,m}f_i^{\text{upper}}\ge
	\max_{i = 1, \ldots, m,\ \bx \in \Omega}\frac{p_i(\bx)}{q_i(\bx)}.
\]

Now, we deal with the achievement function $\psi(\bx)$ over $\Delta$ from the viewpoint of polynomial optimization problems.
For each $\bx\in\mathbb{R}^n$, it holds that
\begin{equation*}
	\begin{aligned}
		\psi(\bx) := & \sup_{\by \in {\Omega}} \min_{i = 1, \ldots, m}
		\left[f_i(\bx) - f_i(\by)\right] \\
       = & \sup_{\by \in {\Omega}} \min_{i = 1, \ldots, m}
		\left[\frac{p_i(\bx)}{q_i(\bx)} - \frac{p_i(\by)}{q_i(\by)}\right] \\
		= & \sup_{\by \in {\Omega}, z \in \R}\ \left\{z \colon  \frac{p_i(\bx)}{q_i(\bx)} - \frac{p_i(\by)}{q_i(\by)} \geq z, \ i = 1, \ldots, m \right\}.
	\end{aligned}
\end{equation*}

For any $\bx\in\Delta$, let
\begin{equation}\label{poly}
	\left\{
	\begin{aligned}
		\tpsi(\bx) := &\ \max_{\by \in \R^n, z \in \R}\ z \\
& \qquad \textrm{s.t.}\quad\ p_i(\bx)q_i(\by) - p_i(\by)q_i(\bx) - z q_i(\bx)q_i(\by) \geq 0, \ i = 1, \ldots, m, \\
& \qquad\qquad\ \  \by \in \Omega,\
z\in[f^{\text{lower}}-f^{\text{upper}}, f^{\text{upper}}-f^{\text{lower}}].\\
	\end{aligned}\right.
\end{equation}
In other words, $\tpsi(\bx)$ over $\Delta$ can be
seen as the optimal value function of the {\it parameter}
polynomial optimization problem \eqref{poly}.
Under ({\bf A1-2}), we have
\begin{proposition}\label{prop::eq}
	$\tpsi(\bx)=\psi(\bx)$ for all $\bx\in\Omega$. Hence$,$
	Propositions \ref{jiao12.28a} and \ref{jiao10.12a} also hold for
	$\tpsi$. 
\end{proposition}
Next, we construct polynomial approximations of $\tpsi$ over
$\Delta$ from above by means of the SDP method proposed in
\cite{Lasserre2015}, and use their sublevel sets to approximate the
set of all weakly ($\be$-)efficient solutions to \eqref{VP}.

Consider the following sets
\begin{equation*}%\label{feasible-set-p}
\K:=\left\{(\bx, \by, z) \in \R^n \times \R^n \times \R\ \colon 
	\left\{
	\begin{aligned}
		&p_i(\bx)q_i(\by) - p_i(\by)q_i(\bx) - z q_i(\bx)q_i(\by) \geq
		0,\\
		& i = 1, \ldots, m,\ \bx\in\Delta,\ \by \in \Omega,\\
		&z\in[f^{\text{lower}}-f^{\text{upper}},
		f^{\text{upper}}-f^{\text{lower}}] 
	\end{aligned}\right.
\right\},
\end{equation*}
and 
\[
	\K_{\bx}:=\left\{(\by, z) \in \mathbb{R}^n\times\mathbb{R}\ \colon\ (\bx,
\by, z)\in\K \right\},\quad\text{for}\ \bx\in\Delta.
\]
Then it is clear that $\K$ is compact and for any $\bx\in\Delta$,
$\tpsi(\bx)=\max_{(\by,z)\in\K_{\bx}} z$.

As proved in \cite[Theorem 1]{Lasserre2015}, a sequence of polynomial
approximations of $\tpsi(\bx)$ on $\Delta$ from above exists mainly due to the
Stone--Weierstrass theorem.
\begin{proposition}{\upshape (c.f. \cite[Theorem 1]{Lasserre2015})}\label{prop::psi} 
	There exists a sequence of polynomials $\{\psi_k \in \R[\bx] \colon
	k \in \N\}$ such that $\psi_k(\bx) \geq \tpsi(\bx)$ for all $\bx
	\in \Delta,$ and $\{\psi_k\}_{k\in\N}$ converges to $\tpsi$ in $L_1(\Delta)$, i.e., 
\begin{align*}
\lim_{k \to \infty} \int_{\Delta}| \psi_k(\bx) - \tpsi(\bx)| d \bx =
0.
\end{align*}
\end{proposition}

Let $\{\psi_k \in \R[\bx] \colon k \in \N\}$ be as in Proposition
\ref{prop::psi}. 
For any $\delta>0$ and $k\in\mathbb{N}$, denote
\[
	\mathcal{A}(\delta, k):=\left\{\bx \in \Omega \colon \psi_k(\bx) \le
\delta\right\}.
\]
For any $\delta>0$, with a slight abuse of notation, we denote
$\Sw^{\delta}:=\Sw^{\be}$, where
$\be=(\delta,\ldots,\delta)$. The following result can be derived
by slightly modifying the proof of \cite[Theorem 3]{Lasserre2015}. It 
shows that we can approximate the set $\mathcal{S}^{\delta}_w$ by the
sequence $\{\mathcal{A}(\delta, k)\}_{k\in\mathbb{N}}$.
%{\color{red}while do not need the Lebesgue measure zero as in \cite[Theorem 3]{Lasserre2015}.}
\begin{theorem}\label{th::appro}
	For any $\delta>0,$ we have $\mathcal{A}(\delta, k)
	\subset \mathcal{S}^{\delta}_w$ and 
%	$\mathcal{A}(\delta, k)\neq\emptyset$
%	for $k$ large enough. Moreover$,$
	\begin{equation}\label{eq::convergence}
\vol\left(\{\bx\in\Omega \colon \psi(\bx)<\delta\}\right)\le 
\lim_{k \to \infty}\mathrm{vol}\left(\mathcal{A}(\delta,
k)\right)\le \mathrm{vol}\left(\left\{\bx \in \Omega \colon \psi(\bx) \le
\delta \right\}\right) = \vol\left(\mathcal{S}^{\delta}_w\right).
\end{equation}
Consequently$,$ if $\vol\left(\{\bx\in\Omega \colon
\psi(\bx)=\delta\}\right)=0,$ then
$\lim_{k\to\infty}\vol\left(\Sw^{\delta}\setminus\mathcal{A}(\delta,
k)\right)=0$. 
\end{theorem}
\begin{proof}
	By Proposition \ref{jiao10.12a}, it is clear that $\mathcal{A}(\delta, k)
	\subset \mathcal{S}^{\delta}_w$. 
By Proposition~\ref{prop::psi}, 
%$\psi_k(\bx) \to \psi(\bx)$ in the $L_1$-norm as $k \to \infty.$ So, 
$\psi_k$ converges to $\tpsi$ in measure, that is, for every $\alpha > 0,$
\begin{align}\label{jiao12.28b}
\lim_{k \to \infty}\mathrm{vol}\left(\{\bx \in \Delta \colon 
| \psi_k(\bx) - \tpsi(\bx)| \ge \alpha \}\right) = 0.
\end{align}
Consequently, for every $\ell \geq 1,$ it holds that
\begin{align*}
&\mathrm{vol}\left(\left\{\bx \in \Omega \colon 
\psi(\bx) \leq \delta + \tfrac{-1}{\ell}\right\}\right)\\
=\ & \mathrm{vol}\left(\left\{\bx \in \Omega \colon 
\tpsi(\bx) \leq \delta + \tfrac{-1}{\ell}\right\}\right)
\tag{by Proposition \ref{prop::eq}}\\
=\ & \mathrm{vol}\left(\left\{\bx \in \Omega \colon 
\tpsi(\bx) \leq \delta + \tfrac{-1}{\ell}\right\} \cap 
\left\{\bx \in \Omega \colon \psi_k(\bx) > \delta \right\} \right) \\
& + \mathrm{vol}\left(\left\{\bx \in \Omega \colon 
\tpsi(\bx) \leq \delta + \tfrac{-1}{\ell}\right\} \cap 
\left\{\bx \in \Omega \colon \psi_k(\bx) \leq \delta \right\}
\right)\\
=\ &\lim_{k \to \infty} \mathrm{vol}\left(\left\{\bx \in \Omega \colon 
\tpsi(\bx) \leq \delta + \tfrac{-1}{\ell}\right\} \cap 
\left\{\bx \in \Omega \colon \psi_k(\bx) \leq \delta \right\} \right) 
\tag{by \eqref{jiao12.28b}}\\
\leq \ & \lim_{k \to \infty} \mathrm{vol}\left( \left\{\bx \in \Omega 
\colon \psi_k(\bx) \leq \delta \right\} \right)\\
\leq\ & \mathrm{vol}\left(\left\{\bx \in \Omega \colon \tpsi(\bx) \le
\delta \right\}\right)\\
=\ &\mathrm{vol}\left(\mathcal{S}^{\delta}_w\right). \tag{by
	Propositions \ref{jiao10.12a} and \ref{prop::eq}}
\end{align*} 
Now, taking $\ell \to \infty$ yields \eqref{eq::convergence} and the
conclusion.
\end{proof}

\begin{corollary}\label{cor::convergence}
	The following assertions are true.
	\begin{enumerate}[\upshape (i)]
	\item For any $d>0,$ there exists $\delta(d)>0$ depending on $d$ such that 
		\[
			\mathcal{A}(\delta, k)\subset \Sw + d\bb
		\]
		holds for any $\delta<\delta(d)$ and any $k\in\mathbb{N}.$
	\item For $d>0$ and any $\delta>0$ with $\vol\left(\{\bx\in\Omega \colon
		\psi(\bx)=\delta\}\right)=0,$ there exists $k(d,
		\delta)\in\mathbb{N}$ depending
		on $\delta$ and $d$ such that
		\[
			\mathcal{S}_w\cap\cl\left(\inte\left(\Omega\setminus\mathcal{S}_w\right)\right)\subset\mathcal{A}(\delta,
		k) + d\bb
	\]
	holds for any $k>k(d, \delta)$. 
\end{enumerate}
\end{corollary}
\begin{proof}
	(i)	Since $\mathcal{A}(\delta, k)\subset\Sw^{\delta}$ for any
	$k\in\mathbb{N}$ by Theorem \ref{th::appro}, the existence of
	$\delta(d)$ is a direct consequence of Proposition
	\ref{prop::dist}. 

(ii) Let
$\bu\in\mathcal{S}_w\cap\cl\left(\inte\left(\Omega\setminus\mathcal{S}_w\right)\right)\neq\emptyset$,
then $\tpsi(\bu)=0$ by Propositions \ref{jiao12.28a} (i) and \ref{prop::eq}, and there
exists a sequence
$\{\bu^{(l)}\}_{l\in\mathbb{N}}\subset\inte\left(\Omega\setminus\mathcal{S}_w\right)$
such that $\lim_{l\to\infty}\bu^{(l)}=\bu$. Fix the numbers $d, \delta>0$. By the
continuity of $\tpsi$ on $\Omega$ (Proposition \ref{prop::eq}), there exists $l_0\in\mathbb{N}$
depending on $d$ and $\delta$ such
that $\tpsi(\bu^{(l_0)})<\delta$ and $\Vert\bu^{(l_0)}-\bu\Vert < d$. 
As $\bu^{(l_0)}\in\inte\left(\Omega\setminus\mathcal{S}_w\right)$, by
the continuity of $\tpsi$ again, there is a neighborhood
$\mathcal{O}^{(l_0)}\subset\Omega$ of $\bu^{(l_0)}$ such
that $\tpsi(\bx)<\delta$ and $\Vert \bx-\bu\Vert<d$ for all
$\bx\in\mathcal{O}^{(l_0)}$. Proposition \ref{jiao10.12a} implies that
$\mathcal{O}^{(l_0)}\subset\Sw^{\delta}$. Then, we show that there
exists $k(d, \delta)\in\mathbb{N}$ such that for any $k>k(d, \delta)$,
it holds that $\mathcal{A}(\delta,
k)\cap\mathcal{O}^{(l_0)}\neq\emptyset$ which means that
$\bu\in\mathcal{A}(d, \delta)+d\bb$ and the conclusion follows. To the
contrary, suppose that such $k(d, \delta)$ does not exist, then there is
subsequence $\{\mathcal{A}(\delta, k_j)\}_{j\in\mathbb{N}}$ with
$k_j\to\infty$ such that $\mathcal{A}(\delta,
k_j)\cap\mathcal{O}^{(l_0)}=\emptyset$ for all $k_j$. Then,
$\mathrm{vol}(\Sw^{\delta}\setminus\mathcal{A}(\delta, k_j))\ge
\mathrm{vol}(\mathcal{O}^{(l_0)})>0$ for all $k_j$. As $\vol\left(\{\bx\in\Omega \colon
		\psi(\bx)=\delta\}\right)=0$, it contradicts the conclusion in
		Theorem \ref{th::main}. 
\end{proof}

\begin{remark}\label{rk::convergence}{\rm From Corollary \ref{cor::convergence} and its
	proof, we can see that
	\begin{itemize}
		\item[{\rm(i)}] If
			$\mathcal{S}_w\cap\cl\left(\inte\left(\Omega\setminus\mathcal{S}_w\right)\right)\neq\emptyset$,
			then for any $\delta>0$, $\mathcal{A}(\delta,
			k)\neq\emptyset$ for $k$ large enough. In fact, we have
			$\mathcal{O}^{(l_0)}\subset\{\bx\in\Omega \colon
			\psi(\bx)<\delta\}$ for the neighborhood
			$\mathcal{O}^{(l_0)}$ in the proof of Corollary
			\ref{cor::convergence}. Then, \eqref{eq::convergence}
			implies that $\mathcal{A}(\delta, k)\neq\emptyset$ for $k$
			large enough. 
		\item[{\rm(ii)}] Suppose there is a sequence
			$\{\delta_i\}_{i\in\mathbb{N}}$ with $\delta_i\downarrow
			0$ such that $\vol\left(\{\bx\in\Omega \colon
			\psi(\bx)=\delta_i\}\right)=0$
	holds for all $i$ and
	$\Sw=\mathcal{S}_w\cap\cl\left(\inte\left(\Omega\setminus\mathcal{S}_w\right)\right)$,
	then Corollary \ref{cor::convergence} (i) and (ii)
	indicate that 
	the whole set of the weakly efficient solutions of
	\eqref{VP} can be approximated arbitrarily well by
	$\mathcal{A}(\delta, k)$ with sufficiently small $\delta>0$ and
	sufficiently large $k\in\mathbb{N}$. 
	\end{itemize}
}
\end{remark}

\subsection{Computational aspects}
%{\color{red}Denote by $L_1(\Delta)$ the Lebesgue space of measurable functions $\varphi: \Delta \to \R$ that are integrable with respect to the Lebesgue measure on $\Delta,$ i.e., such that $\int_{\Delta} |\varphi| d y < \infty.$}
%
Now we follow the scheme proposed in \cite[Section 3.3]{Lasserre2015}
to construct a sequence of polynomials
$(\psi_k)_{k\in\mathbb{N}}\in\mathbb{R}[\bx]$ as defined in
Proposition \ref{prop::psi}.  

We denote the following $m+r+2n+1$ polynomials in $\mathbb{R}[\bx,\by,z]$
\begin{equation}\label{eq::h}
	\begin{aligned}
		&h_{1,1}(\bx,\by,z)=p_1(\bx)q_1(\by) - p_1(\by)q_1(\bx) - z q_1(\bx)q_1(\by), \ldots, \\
		&h_{1,m}(\bx,\by,z)=p_m(\bx)q_m(\by) - p_m(\by)q_m(\bx) - z q_m(\bx)q_m(\by), \\
		&h_{2,1}(\bx,\by,z)=g_1(\by), \ldots, h_{2,r}(\bx,\by,z)=g_r(\by),\\
		&h_{2,r+1}(\bx,\by,z)=1-y_1^2, \ldots, h_{2,r+n}(\bx,\by,z)=1-y_n^2,\\
		&h_{3,1}(\bx,\by,z)=1-x_1^2, \ldots,
		h_{3,n}(\bx,\by,z)=1-x_n^2,\\
		&h_{4,1}(\bx,\by,z)=(f^{\text{upper}}-f^{\text{lower}})^2-z^2.
	\end{aligned}
\end{equation}
Denote by
$J_1=\{1,\ldots,m\}$, $J_2=\{1,\ldots,r+n\}$, $J_3=\{1,\ldots,n\}$ and
$J_4=\{1\}$. Then, 
\[
	\K=\{ (\bx, \by, z) \in \R^n \times \R^n \times \R\ \colon \
		h_{i,j}(\bx,\by,z)\ge 0,\ i=1,\ldots,4,\ j\in J_i\}. 
\]

Let $\lambda$ be the scaled Lebesgue measure on $\Delta$, i.e.,
$d\lambda(\bx)=d\bx/2^n$, and 
\[
	\gamma_{\alpha}:=\int_{\Delta} \bx^{\alpha}
	d\lambda(\bx)=\left\{
		\begin{aligned}
			&0,	\quad\text{if $\alpha_i$ is odd for some } i\\
			&\prod_{i=1}^n(\alpha_i+1)^{-1},\quad\text{otherwise}
		\end{aligned}
		\right.
\]
be the moment of $\lambda$ for each $\alpha\in\mathbb{N}^n$.

For each $k \in \N,$ with $k \ge  \max\left\{\lceil \tfrac{\deg
	h_{i,j}}{2}\rceil,\ i=1,\ldots, 4,\ j\in J_i\right\}$, consider the following
optimization problem,
\begin{equation}\label{sdp}
	\left\{
	\begin{aligned}
		\rho^*_k:=\inf_{\phi, \sigma_0, \sigma_{i,j}}\
		&\int_{\Delta}\phi(\bx)d\lambda(\bx)\left(=\sum_{\alpha \in \N^n_{2k}} c_{\alpha}
		\gamma_{\alpha}\right)\\
		\text{s.t.\quad}&\phi(\bx)=\sum_{\alpha \in \N^n_{2k}}
		c_{\alpha}\bx^{\alpha}\in\mathbb{R}[\bx]_{2k},\ c_{\alpha}\in\mathbb{R},\\
		&\phi(\bx)-z = \sigma_0 + \sum_{i=1}^4\sum_{j \in J_i}
		\sigma_{i,j} h_{i,j},\ \sigma_0, \sigma_{i,j} \in\Sigma^2[\bx, \by, z],\\ 
		&\deg(\sigma_0), \deg(\sigma_{i,j} h_{i,j}) \leq 2k,\
		i=1,\ldots,4, \ j \in J_i,\\
	\end{aligned}\right. \tag{P$_k$}
\end{equation}
which can be reduced to an SDP problem (c.f. \cite{Lasserre2001,Lasserre2010}).
Clearly, for any $\left(\phi, \sigma_0, \sigma_{i,j}\right)$ feasible
to \eqref{sdp}, we have $\phi(\bx)\ge \tpsi(\bx)$ on $\Delta$. 
The following result follows directly from \cite[Theorem
5]{Lasserre2015} and we include here a brief proof for the sake of
completeness. It shows that we can compute the sequence of polynomials
$\{\psi_k \in \R[\bx] \colon k \in \N\}$ in Proposition
\ref{prop::psi} by solving \eqref{sdp}. 

\begin{theorem} \label{th::main}
	We have
	$\lim_{k\rightarrow\infty}\rho^*_k=\int_{\Delta}\tpsi(\bx)d\lambda(\bx)$. 
%Suppose that $\Omega$ has a nonempty interior$,$ then there is no dual gap
%between \eqref{sdp} and \eqref{dsdp}, both of which have nonempty
%optimal solution set. 
	Consequently$,$ let $\left(\psi_k, \sigma_0^{(k)}, \sigma^{(k)}_{i,j}\right)$ be a nearly optimal
solution to \eqref{sdp}$,$ e.g.$,$ $\int_{\Delta}\psi_{k}d\lambda(\bx)\le\rho^*_k+1/k,$
%and
%	$\psi^*_k(\bx)=\sum_{\alpha\in\mathbb{N}^n_{2k}}
%c^*_{\alpha}\bx^{\alpha},$ 
then $\psi_k(\bx)\ge \tpsi(\bx)$ on $\Delta$ and 
\[
\lim_{k \to \infty} \int_{\Delta}| \psi_k(\bx) - \tpsi(\bx)| d\lambda(\bx) = 0.
\]
\end{theorem}

\begin{proof} We only need to prove that
	$\lim_{k\rightarrow\infty}\rho^*_k=\int_{\Delta}\tpsi(\bx)d\lambda(\bx)$.
Consider the following infinite-dimensional linear program 
\[
	\left\{
	\begin{aligned}
		\rho^*:=\inf_{\phi}\
		&\int_{\Delta}\phi(\bx)d\lambda(\bx)\left(=\sum_{\alpha \in \N^n_{2k}} c_{\alpha}
		\gamma_{\alpha}\right)\\
		\text{s.t.\ }&\phi(\bx)=\sum_{\alpha \in \N^n}
		c_{\alpha}\bx^{\alpha}\in\mathbb{R}[\bx],\ c_{\alpha}\in\mathbb{R},\\
		&\phi(\bx)-z \ge 0, \quad \forall\ (\bx, \by, z)\in\K.
	\end{aligned}\right.
\]
It is clear that $\Delta,$ $\K$ are
compact and $\K_{\bx}$ is nonempty for every $\bx\in\Delta$. Then, by
\cite[Corollary 2.6]{Lasserre2015}, it holds that
$\rho^*=\int_{\Delta}\tpsi(\bx)d\lambda(\bx)$. Let
$(\phi_{\ell})_{\ell\in\mathbb{N}}$ be a minimizing sequence of the
above problem.  
For any $\ell\in\mathbb{N}$, let
$\phi'_{\ell}(\bx)=\phi_{\ell}(\bx)+1/{\ell}$, then we
have $\phi'_{\ell}(\bx)-z \ge 1/{\ell}>0$ on $\K$. Notice that
	\[
		2n+(f^{\text{upper}}-f^{\text{lower}})^2-\sum_{i=1}^n\left(x_i^2+y_i^2\right)-z^2=\sum_{j=r+1}^{r+n}h_{2,j}+\sum_{j=1}^{n}h_{3,j}+h_{4,1},
	\]
	that is, Assumption \ref{Archimedean} holds for the defining
	polynomials of $\K$. Therefore, by
	Putinar's Positivstellensatz (Theorem \ref{th::putinar}), there
	exists $k_{\ell}\in\mathbb{N}$ and
$\sigma_0^{(\ell)}, \sigma_{i,j}^{(\ell)} \in\Sigma^2[\bx, \by, z]$
such that $(\phi'_{\ell}, \sigma_0^{(\ell)}, \sigma_{i,j}^{(\ell)})$ is
a feasible solution to (P$_{k_{\ell}}$). Note that $\rho^*\le\rho^*_k$
holds for any $k\in\mathbb{N}$. Then, it implies that
\[
\int_{\Delta}\tpsi(\bx)d\lambda(\bx)=\rho^*\le\rho^*_{k_\ell}\le \int_{\Delta}\phi_{\ell}(\bx)d\lambda(\bx)+\frac{1}{\ell}\ \downarrow
		\rho^* =\int_{\Delta}\tpsi(\bx)d\lambda(\bx).
\]
As $\rho^*_k$ is monotone, we have
$\lim_{k\rightarrow\infty}\rho^*_k=\int_{\Delta}\tpsi(\bx)d\lambda(\bx)$.
\end{proof}

Next, we propose a sparse version of the SDP problem \eqref{sdp} by
exploiting its sparsity pattern, which reduces the computational costs
at the order $k$.
Add a redundant polynomial
\begin{equation*}\label{eq::redundant}
	h_{1,m+1}(\bx,\by,z)=2n+(f^{\text{upper}}-f^{\text{lower}})^2-\sum_{i=1}^n\left(x_i^2+y_i^2\right)-z^2
\end{equation*}
in \eqref{eq::h} and reset $J_1=\{1, \ldots, m+1\}$. 
Denote the following subsets of variables
$I_1=\{\bx, \by, z\}$, $I_2=\{\by\}$, $I_3=\{\bx\}$ and $I_4=\{z\}$.
For $i=1,\ldots,4$, denote by $\mathbb{R}[I_i]$ the ring of real
polynomials in the variables in $I_i$. Then, the following conditions hold.
\begin{enumerate}[(i)]
	\item  For each $i=1, 2, 3$, there exists some $s\le i$ such that
		$I_{i+1}\cap\bigcup_{j=1}^i I_j\subseteq I_s$;
	\item For each $i=1,\ldots,4$, and each $j\in J_i$,
		$h_{i,j}\in\mathbb{R}[I_i]$;
\item $\sum_{\alpha\in\mathbb{N}^n_{2k}} c_{\alpha}\bx^{\alpha} -z$ in
	\eqref{sdp} is the difference of two polynomials in
	$\mathbb{R}[I_3]$ and $\mathbb{R}[I_4]$, respectively.
\end{enumerate} 

Then, by the sparse version of Putinar's
Positivstellensatz (Theorem \ref{th::sputinar}), we can construct a
sparse version of \eqref{sdp} as 
\begin{equation}\label{ssdp}
	\left\{
	\begin{aligned}
		\tilde{\rho}^*_k:=\inf_{\phi, \sigma_{i,0}, \sigma_{i,j}}\
		&\int_{\Delta}\phi(\bx)d\lambda(\bx)\left(=\sum_{\alpha \in \N^n_{2k}} c_{\alpha}
		\gamma_{\alpha}\right)\\
		\text{s.t.\quad }&\phi(\bx)=\sum_{\alpha \in \N^n_{2k}}
		c_{\alpha}\bx^{\alpha}\in\mathbb{R}[\bx]_{2k},\ c_{\alpha}\in\mathbb{R},\\
		&\phi(\bx)-z = \sum_{i=1}^4\left(\sigma_{i,0} + \sum_{j \in J_i}
		\sigma_{i,j} h_{i,j}\right),\ \sigma_{i, 0}, \sigma_{i,j} \in\Sigma^2[I_i],\\ 
		&\deg(\sigma_{i, 0}), \deg(\sigma_{i,j} h_{i,j}) \leq 2k,\
		i=1,\ldots,4, \ j \in J_i.\\
	\end{aligned}\right. \tag{SP$_k$}
\end{equation}
\begin{theorem}
	The statements for \eqref{sdp} in Theorem \ref{th::main} also
	hold for \eqref{ssdp}.
\end{theorem}
\begin{proof}
	Let $\phi'_{\ell}$ be the polynomial in the proof of Theorem
	\ref{th::main}. Note that Assumption \ref{assu::sparsity} holds by
	adding the redundant
	polynomial $h_{1,m+1}$. Then, by Theorem \ref{th::sputinar}, there exists
$\tilde{k}_{\ell}\in\mathbb{N}$ and
$\sigma_{i, 0}^{(\ell)}, \sigma_{i,j}^{(\ell)} \in\Sigma^2[I_i]$,
$i=1,\ldots,4, \ j \in J_i$
such that $\left(\phi'_{\ell}, \sigma_{i,0}^{(\ell)}, \sigma_{i,j}^{(\ell)}\right)$ is
a feasible solution to (SP$_{\tilde{k}_{\ell}}$). Hence, the
conclusion follows from the proof of Theorem \ref{th::main}.
\end{proof}

\subsection{Comparisons with existing SDP relaxation
methods}\label{sec::compare}
Now, we compare our method with the recent existing work in
\cite{Nie2021} and \cite{Magron2014}. All the three methods can deal with
vector (nonlinear) polynomial optimization problems by SDP
relaxations, without convexity assumptions on the involved functions.
For convenience, we assume that all objectives $f_i$'s in \eqref{VP}
are polynomials, i.e., $q_i(\bx)=1$, $i=1, \ldots, m$.

To get weakly efficient solutions to \eqref{VP}, Nie and
Yang~\cite{Nie2021} used the linear scalarization and the Chebyshev
scalarization techniques to scalarize \eqref{VP} to a single objective
polynomial optimization problem and solve it by the SDP relaxation
method proposed in \cite{Nie2019}. Precisely, for a given nonzero 
weighting parameter
$w:=(w_1, \ldots, w_m)\in\mathbb{R}^m$, the linear scalarization
scalarizes the problem \eqref{VP} to 
\begin{equation}\label{eq::ls}
	\min\ w_1f_1(\bx)+\cdots+w_mf_m(\bx) \quad\text{s.t. }\ \bx\in\Omega,
\end{equation}
and the Chebyshev scalarization scalarizes the problem \eqref{VP} to 
\begin{equation}\label{eq::cs}
	\min_{\bx\in\Omega} \max_{1\le i\le m} w_i(f_i(\bx) -
	f_i^*),
\end{equation}
where each $f_i^*$ is the goal which decision maker wants to achieve for the
objective $f_i$. In general, by the scalarizations \eqref{eq::ls} and
\eqref{eq::cs}, we can only
find one or some particular (weakly) efficient solutions for a given
weight $w$. Moreover, a serious drawback of linear
scalarization is that it can not provide a solution among sunken parts
of Pareto frontier due to ``duality gap'' of nonconvex cases (see Example \ref{ex::3}). 
Instead, the sets $\{\mathcal{A}(\delta, k)\}$ computed by our
method can approximate the whole set of weakly efficient solutions in
some sense under certain conditions. The representation of
$\mathcal{A}(\delta, k)$ as the intersection of the sublevel set of a single
polynomial and the feasible set is more desirable in some
applications. For example, it can be used in optimization problems
with Pareto constraints (c.f.  \cite{HP2002}). A Pareto constraint can
be replaced by the polynomial inequality $\psi_{k}(\bx)\le \delta$
with small $\delta>0$ and large $k\in\mathbb{N}$ (see Example \ref{ex::3}).

On the other hand, Magron et al. \cite{Magron2014} studied the problem \eqref{VP} with
$m=2$. Rather than computing the weakly efficient solutions, they
presented a method to approximate as closely as desired the Pareto
curve which is the image of the objective functions over the set of
weakly efficient solutions. To this end, they also considered the
scalarizations \eqref{eq::ls} and \eqref{eq::cs}, as well as the
parametric sublevel set approximation method which is inspired by
\cite{Gorissen2012} and amounts to solving the following parametric problem
\begin{equation}\label{eq::ss}
	\min_{\bx\in\Omega}\ f_2(\bx) \quad\text{s.t. }\  f_1(\bx) \le w,
\end{equation}
with a parameter $w\in[\min_{\bx\in\Omega} f_1(\bx), \max_{\bx\in\Omega}
f_1(\bx)]$. By treating $w$ in \eqref{eq::ls}, \eqref{eq::cs} and
\eqref{eq::ss} as a parameter and employing the ``joint+marginal''
approach proposed in \cite{Lasserre2010}, they associated each
scalarization problem a
hierarchy of SDP relaxations and obtained an approximation of
the Pareto curve by solving an inverse problem (for \eqref{eq::ls} and
\eqref{eq::cs}) or by building a polynomial underestimator (for
\eqref{eq::ss}). Again, comparing with the approximate Pareto curve obtained in
\cite{Magron2014}, it is more convenient to apply our explicit
approximation $\mathcal{A}(\delta, k)$ of the weakly efficient
solution set to optimization problems with Pareto constraint. 
Moreover, when using the scalarization problems \eqref{eq::ls} and
\eqref{eq::cs}, the approach in \cite{Magron2014} requires that for
almost all the values of the parameter $w$, these parametric
problems \eqref{eq::ls} and \eqref{eq::cs} have a unique global
minimizer. Namely, there should be a
one-to-one correspondence between the points on the computed Pareto
curves and the associated weakly efficient solutions in the feasible
set. Note that our method does not have such restriction when
approximating the set of weakly efficient solutions (see Example
\ref{ex::5}).

\subsection{Numerical experiments}
Here we present some numerical examples to illustrate
the behavior of the sets $\mathcal{A}(\delta, k)$ in
approximating $\Sw$ as $\delta\to 0$ and
$k\to\infty$. 
We use the software {\sf Yalmip} \cite{YALMIP} to implement the
problems \eqref{ssdp} and call the SDP solver SeDuMi \cite{Sturm99} to
solve the resulting SDP problems. 
For the examples with $m=2$, to show how close the sets
$\mathcal{A}(\delta, k)$ in approximating $\Sw$, we
illustrate the corresponding images of $f(\Omega)$ and
$f(\mathcal{A}(\delta, k))$. To
this end, we choose a square containing $\Omega$. For each point
$u$ on a uniform discrete grid inside the square, we check if
$u\in\Omega$ (resp., $u\in\mathcal{A}(\delta, k)$).  If so,
we have $(f_1(u),f_2(u))\in f(\Omega)$ (resp.,
$(f_1(u),f_2(u))\in f(\mathcal{A}(\delta, k))$) and we plot the point
$(f_1(u),f_2(u))$ in grey (resp., in red) in the image plane. 

\begin{example}\label{ex::1}
	{\rm Consider the problem
\[	
\left\{\begin{aligned}
{\rm Min}_{\mathbb{R}^3_+} &\ \left(x_1,\ x_2,\ x_1^2+x_2^2\right)\\
\text{s.t.}\quad &\ \bx\in\Omega_1:=\{\bx\in\mathbb{R}^2 : x_1^2 +
x_2^2 \le 1\}.
\end{aligned}
\right.
\]
Clearly, the set of all weakly efficient solution to this problem is
\[\mathcal{S}_w = \left\{\bx \in \R^2 \colon x_1 \leq 0, x_2 \leq 0, x^2_1 + x^2_2 \leq 1 \right\}.
\]
For any $\delta>0$, by considering the four quadrants of $\mathbb{R}^2$
one by one, it is easy to check by definition that the set
$\Sw^{\delta}$ consists of the following four sets
\[
	\begin{aligned}
		&\{\bx\in\mathbb{R}^2 \colon x_1\ge\delta,\ x_2\ge\delta,\
	x_1^2+x_2^2\le \delta\},	\\
		&\{\bx\in\mathbb{R}^2 \colon x_1\le\delta,\ x_2\ge\delta,\
	x_2^2+2\delta x_1-\delta-\delta^2\le 0,\ x_1^2+x_2^2\le 1\},	\\
		&\{\bx\in\mathbb{R}^2 \colon x_1\le\delta,\ x_2\le\delta,\
	x_1^2+x_2^2\le 1\},	\\
		&\{\bx\in\mathbb{R}^2 \colon x_1\ge\delta,\ x_2\le\delta,\
	x_1^2+2\delta x_2-\delta-\delta^2\le 0,\ x_1^2+x_2^2\le 1\}.	\\
	\end{aligned}
\]
\begin{figure}
	\centering
	\subfigure[$\Sw^{\delta}$]{
\scalebox{0.4}{
	\includegraphics[trim=80 200 80 200,clip]{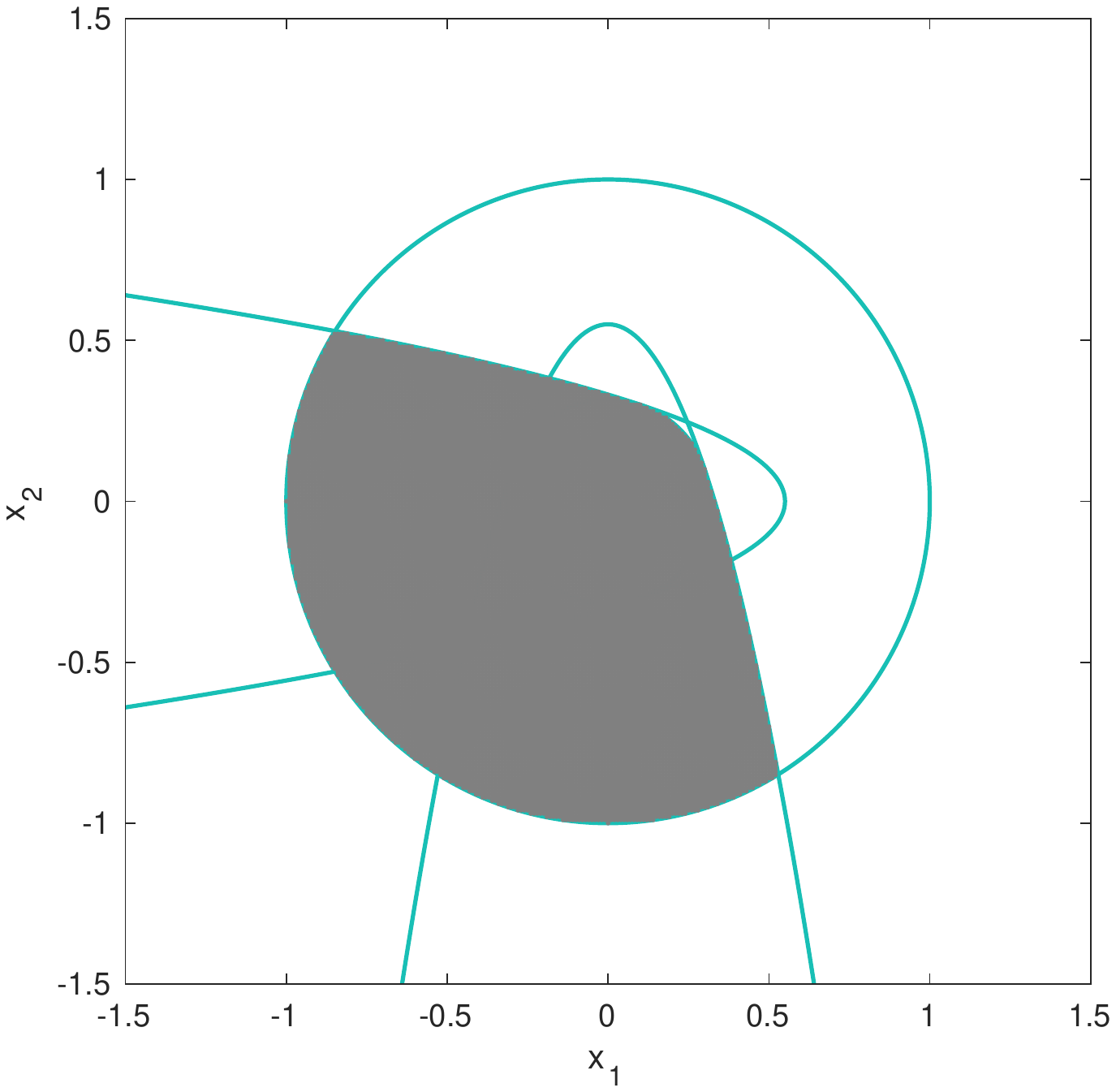}
}
}
	\subfigure[$\mathcal{A}(\delta, 2)$]{
\scalebox{0.4}{
	\includegraphics[trim=80 200 80 200,clip]{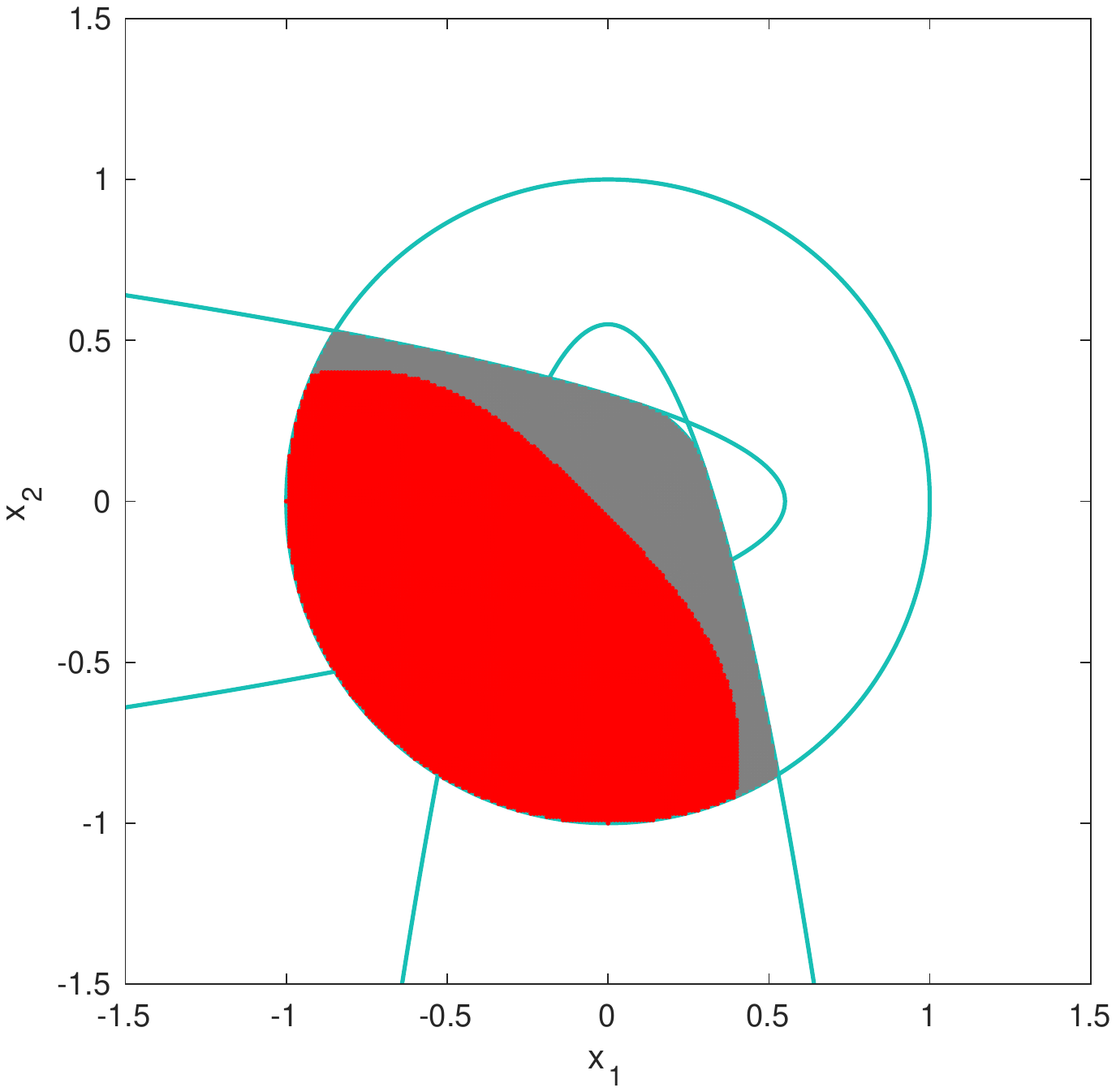}
}
}\\
	\subfigure[$\mathcal{A}(\delta, 3)$]{
\scalebox{0.4}{
	\includegraphics[trim=80 200 80 200,clip]{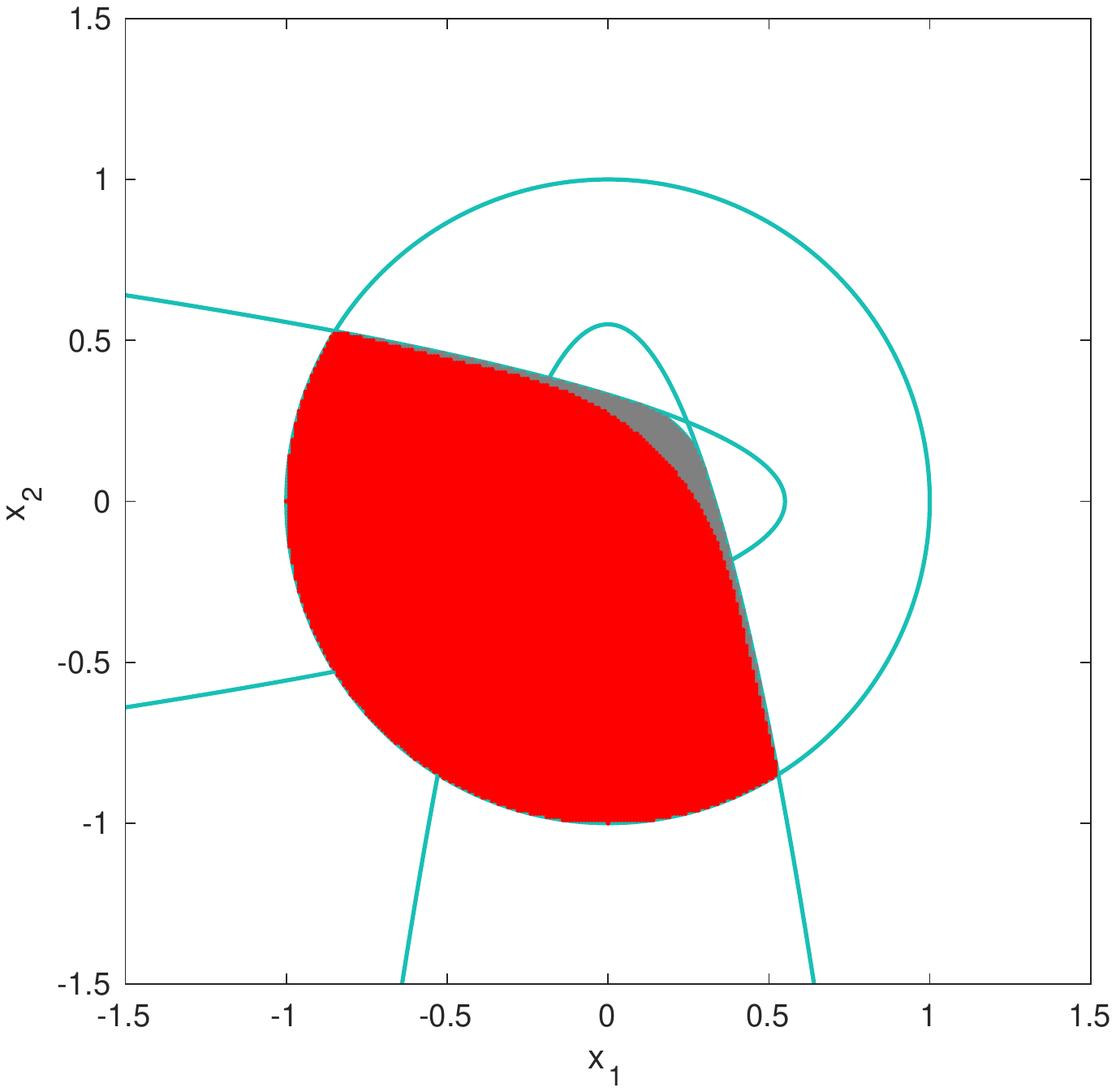}
}
}
	\subfigure[$\mathcal{A}(\delta, 4)$]{
\scalebox{0.4}{
	\includegraphics[trim=80 200 80 200,clip]{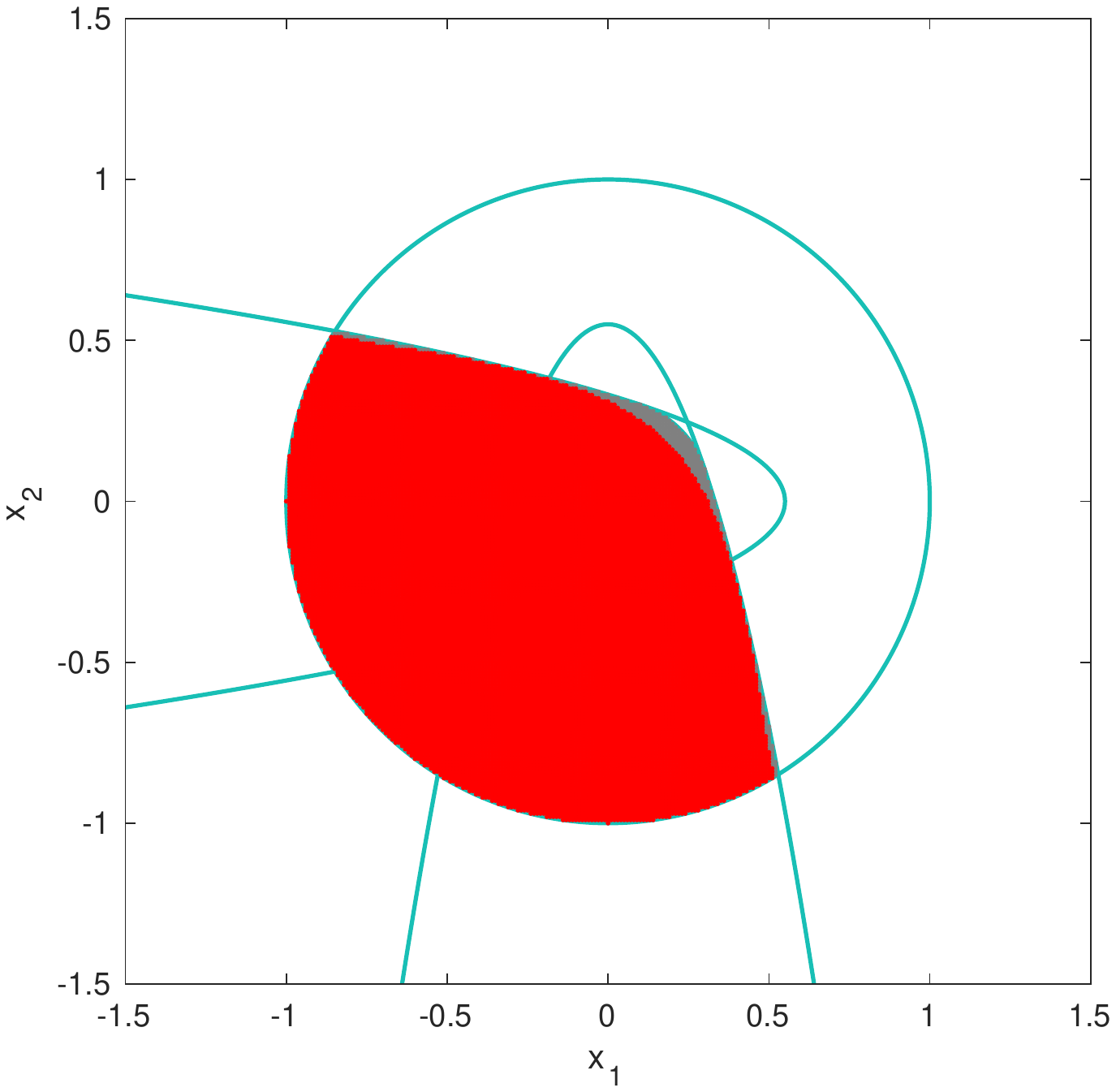}
}
}
\caption{The set $\Sw^{\delta}$ and its approximations
$\mathcal{A}(\delta, k)$ with $\delta=0.1$, $k=2, 3, 4$, in Example \ref{ex::1}.}
\label{fig-1}
\end{figure}
For $\delta=0.1$, we show the set $\Sw^{\delta}$ and its
approximations $\mathcal{A}(\delta, k)$, $k=2, 3, 4$, in Figure
\ref{fig-1}.\qed
}
\end{example}

\begin{example}\label{ex::2}{\rm
To illustrate how the set $\mathcal{A}(\delta, k)$ behaves in
approximating the set of weakly efficient solutions $\Sw$ as
$\delta\to 0$ and $k\to\infty$, we consider the problem
\[	
	\left\{
\begin{aligned}
	{\rm Min}_{\mathbb{R}^2_+} &\ \left(x_1,\
	\frac{x_2^2-2x_1x_2+1}{x_2^2+1}\right)\\
	\text{s.t.}\quad &\ \bx\in\Omega_2:=\left\{\bx\in\mathbb{R}^2 \colon 1 - x_1^2 - x_2^2 \geq 0 \right\}.
\end{aligned}
\right.
\]
We plot the images $f(\mathcal{A}(\delta, k))$ with $\delta=0.1, 0.05,
0.02$ and $k=3, 4, 5$, as well as $f(\Omega_2)$, in Figure
\ref{fig-2}. \qed
%Observe that the sets $\Sw$ and $\Sw^{\delta}, \delta>0$ are convex
%and disconnected.
\begin{figure}[h]
	\centering
	\subfigure[$f(\mathcal{A}(0.1,3))$]{
\scalebox{0.3}{
	\includegraphics[trim=60 200 80 200,clip]{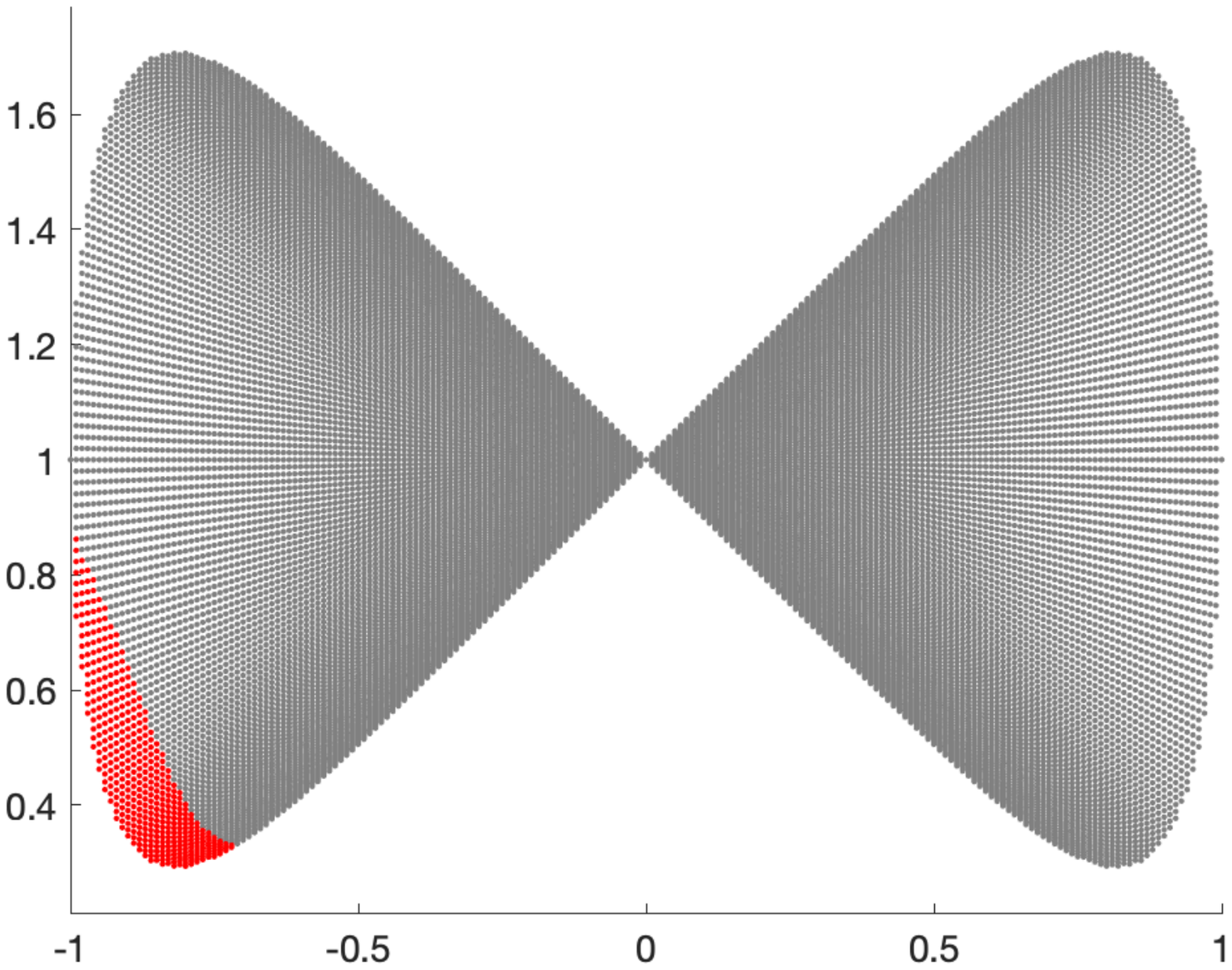}
}
}
	\subfigure[$f(\mathcal{A}(0.1,4))$]{
\scalebox{0.3}{
	\includegraphics[trim=60 200 80 200,clip]{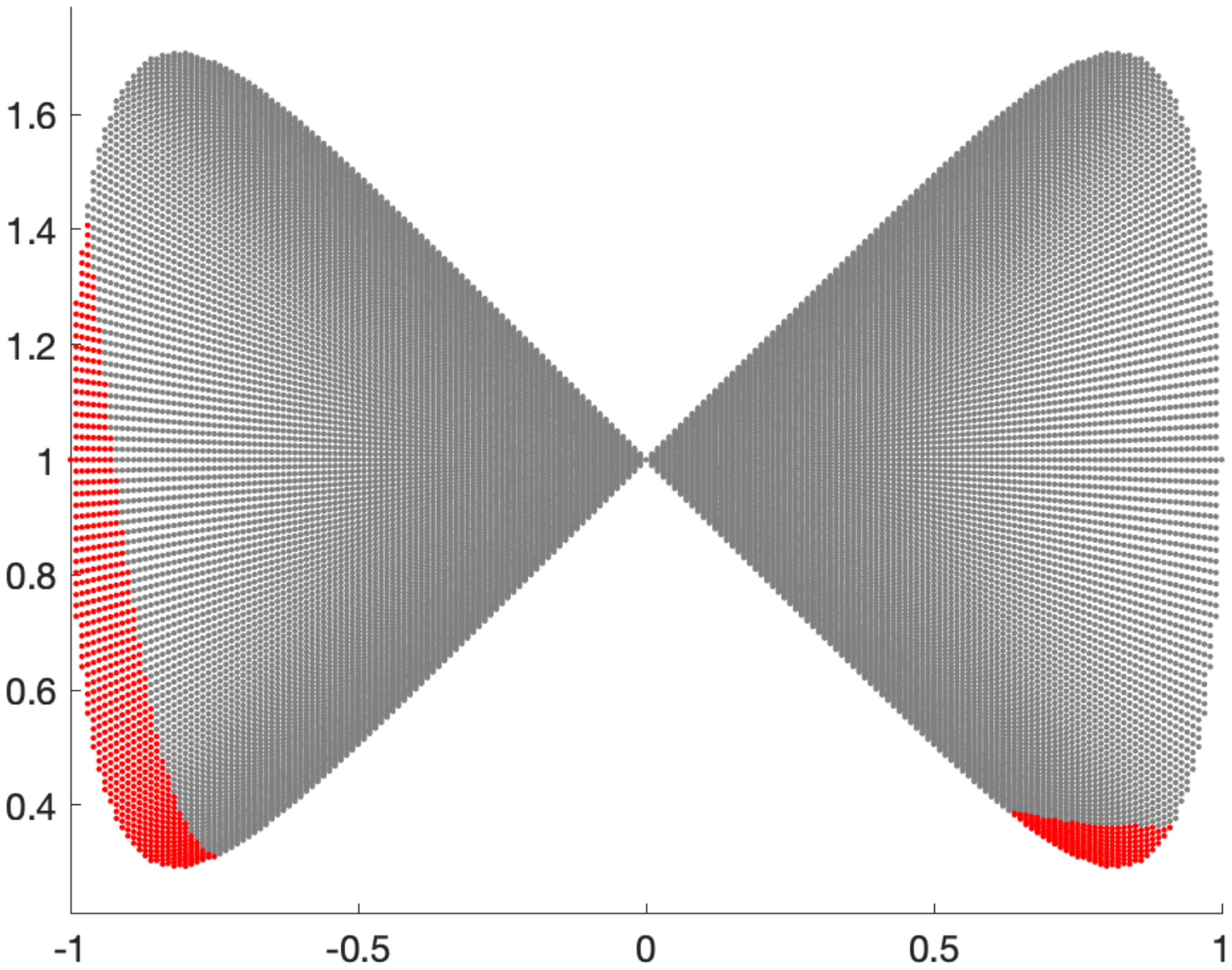}
}
}
	\subfigure[$f(\mathcal{A}(0.1,5))$]{
\scalebox{0.3}{
	\includegraphics[trim=60 200 80 200,clip]{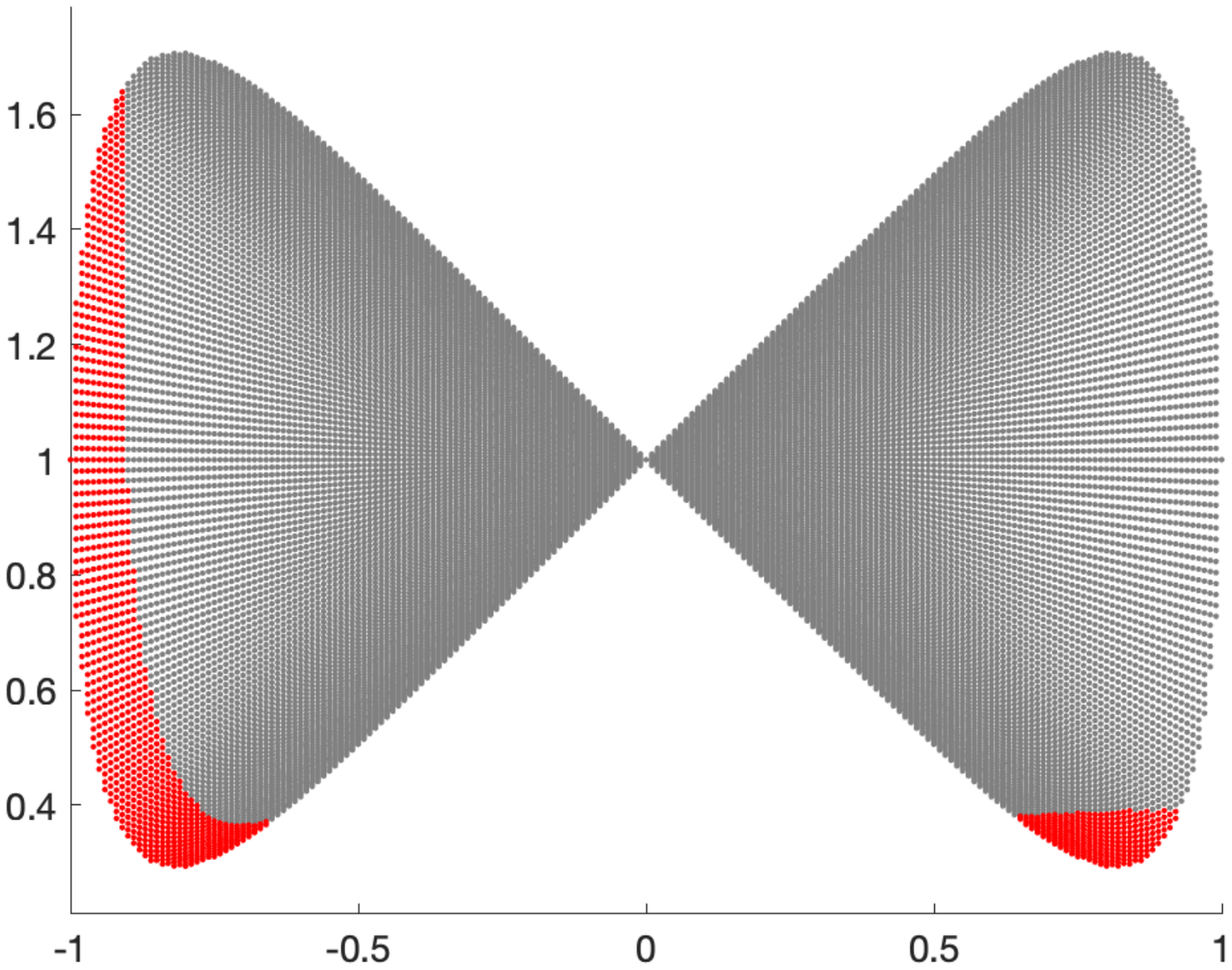}
}
}\\
\subfigure[$f(\mathcal{A}(0.05,3))$]{
\scalebox{0.3}{
	\includegraphics[trim=60 200 80 200,clip]{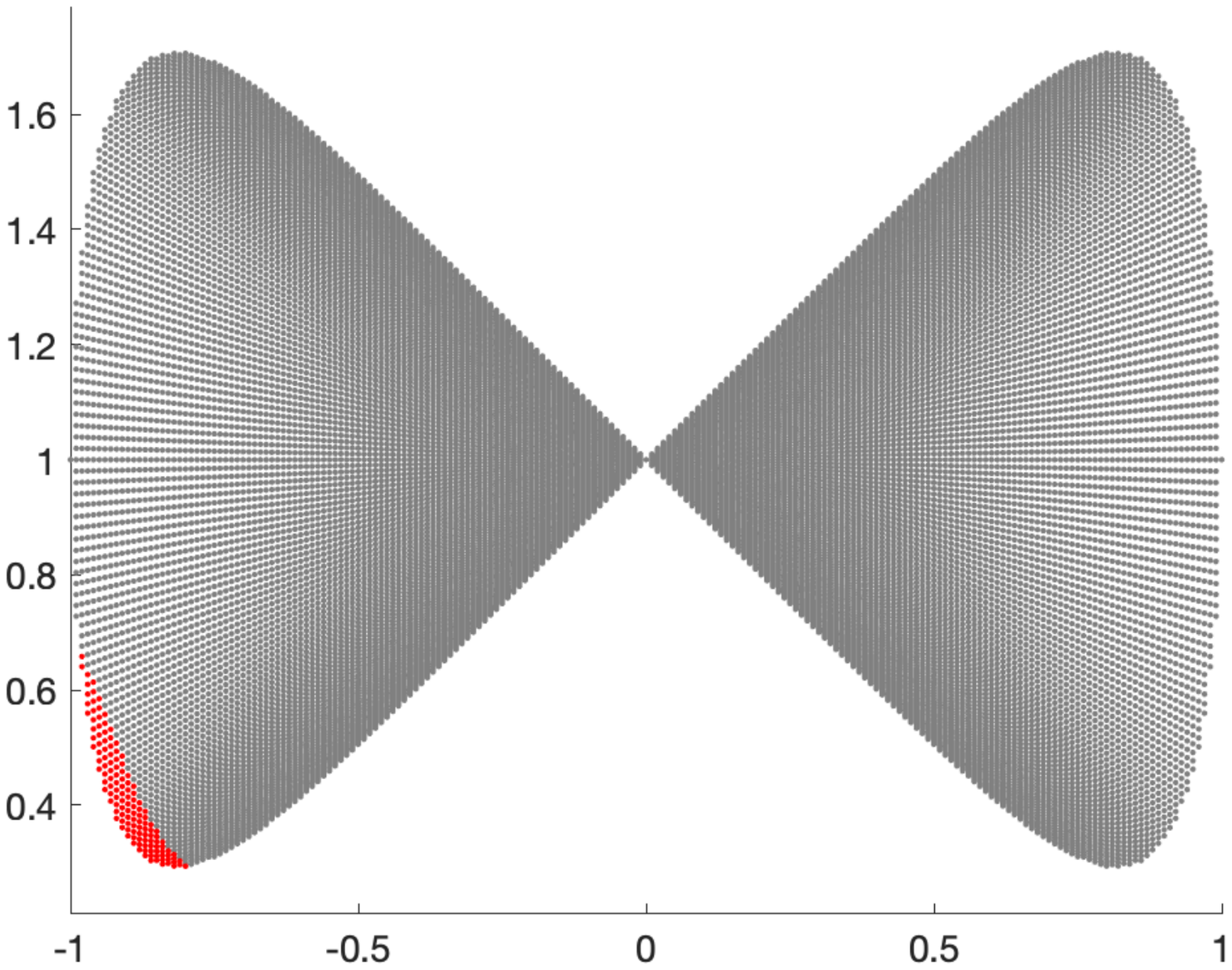}
}
}
\subfigure[$f(\mathcal{A}(0.05,4))$]{
\scalebox{0.3}{
	\includegraphics[trim=60 200 80 200,clip]{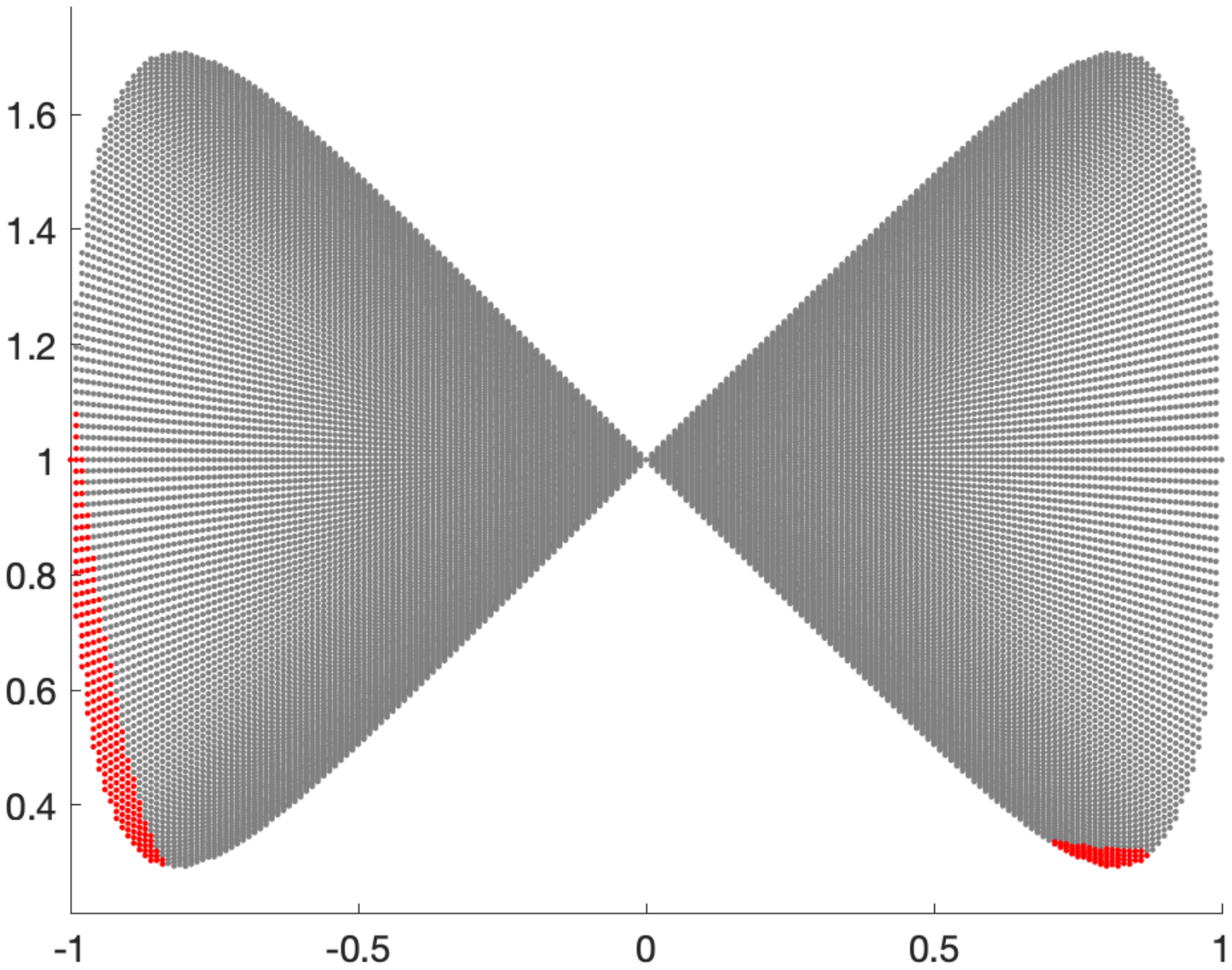}
}
}
\subfigure[$f(\mathcal{A}(0.05,5))$]{
\scalebox{0.3}{
	\includegraphics[trim=60 200 80 200,clip]{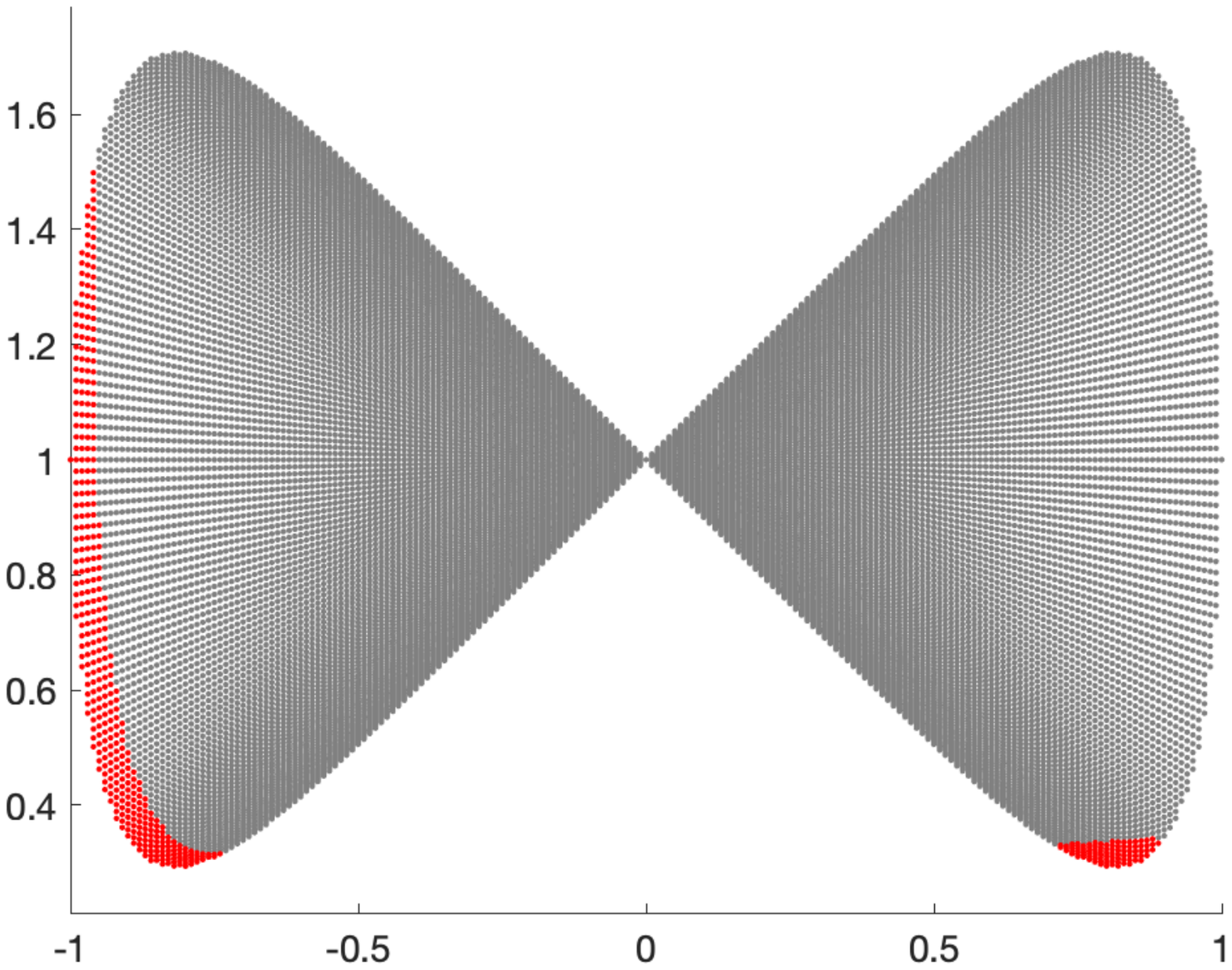}
}
}\\
\subfigure[$f(\mathcal{A}(0.02,3))$]{
\scalebox{0.3}{
	\includegraphics[trim=60 200 80 200,clip]{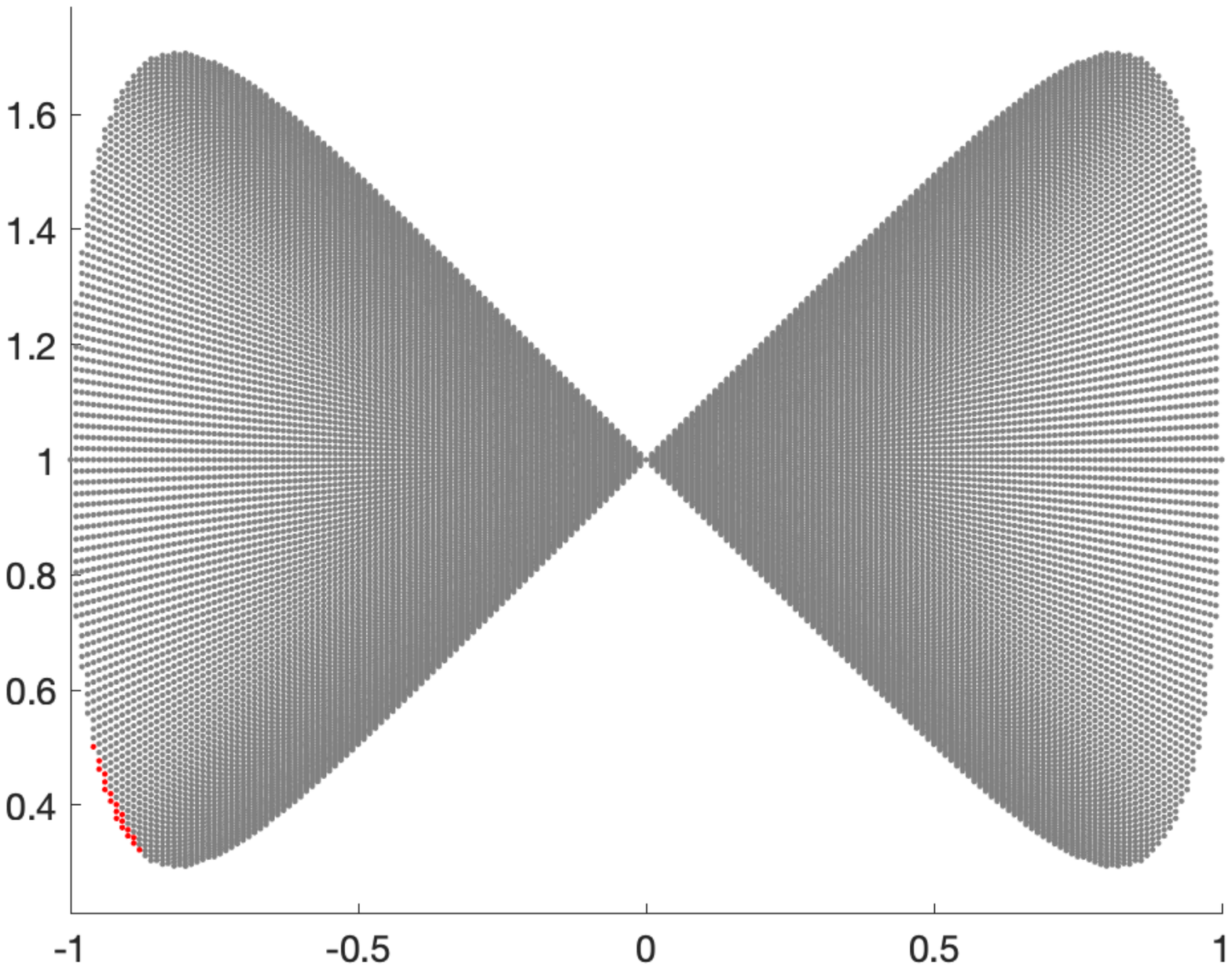}
}
}
\subfigure[$f(\mathcal{A}(0.02,4))$]{
\scalebox{0.3}{
	\includegraphics[trim=60 200 80 200,clip]{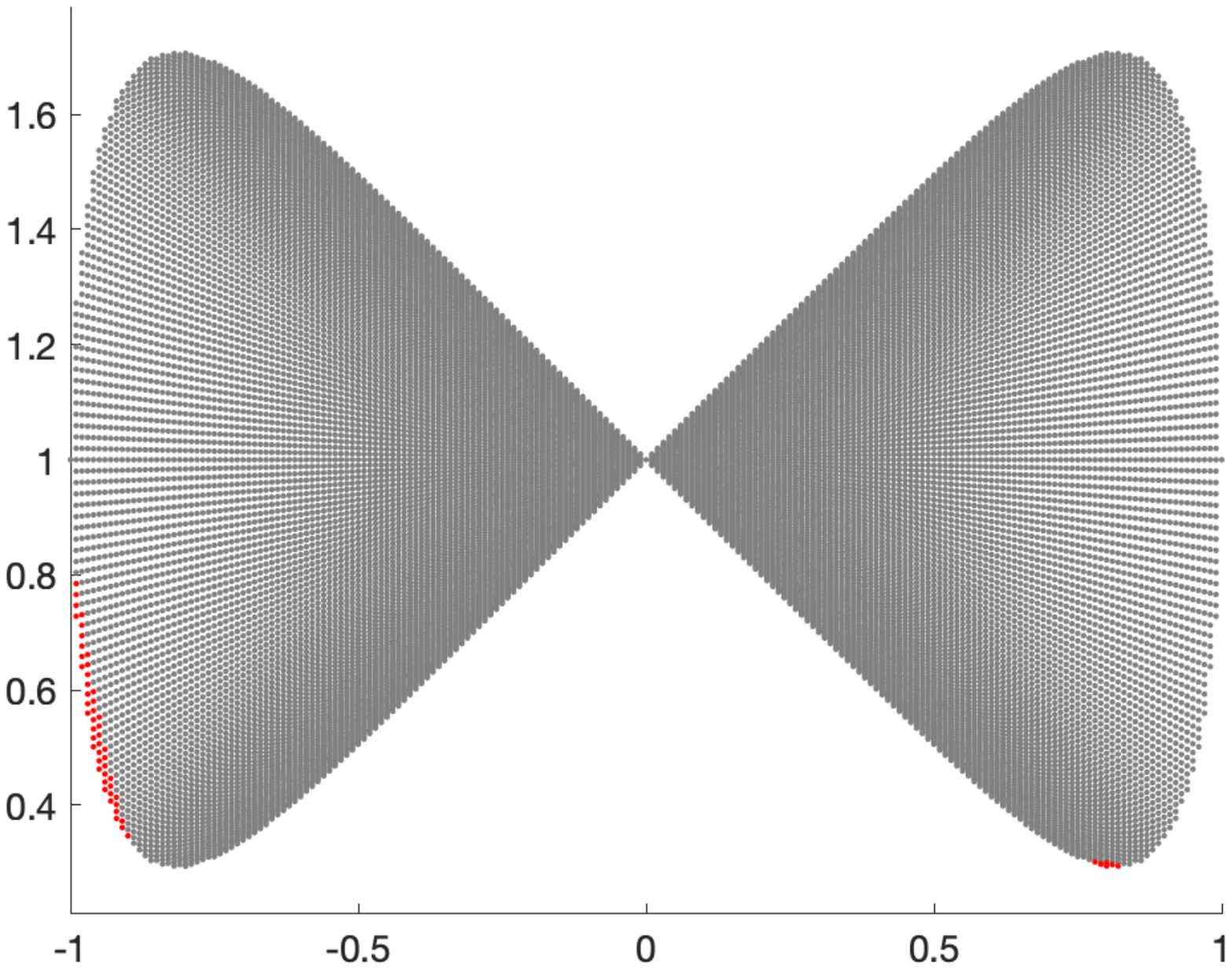}
}
}
\subfigure[$f(\mathcal{A}(0.02,5))$]{
\scalebox{0.3}{
	\includegraphics[trim=60 200 80 200,clip]{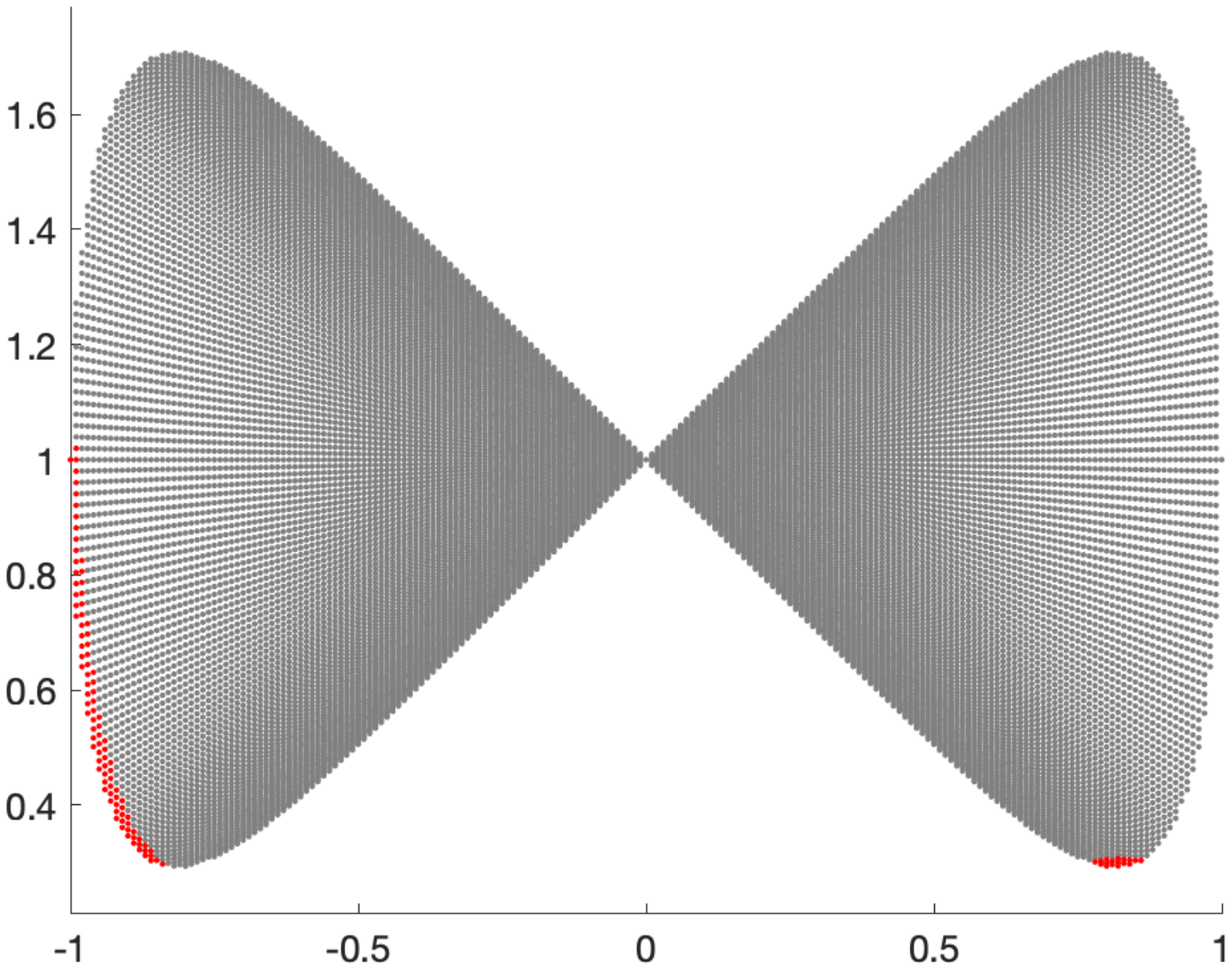}
}
}
\caption{The images $f(\mathcal{A}(\delta, k))$ (in red) and
$f(\Omega_2)$ (in gray) in Example \ref{ex::2}}
\label{fig-2}
\end{figure}
}
\end{example}

\begin{example}\label{ex::3}{\rm
Consider the problem   
\[	
	\left\{
\begin{aligned}
	{\rm Min}_{\mathbb{R}^2_+} &\ \left(\frac{\sqrt{2}}{2}(-x_1+x_2),\
	\frac{\sqrt{2}}{2}(x_1+x_2)\right)\\
	\text{s.t.}\quad &\ \bx\in\Omega_3:=\left\{\bx\in\mathbb{R}^2
	\colon  g(\bx) :=x_2^2(1-x_1^2)-(x_1^2+2x_2-1)^2 \ge 0 \right\},
\end{aligned}
\right.
\]
In fact, the equality $g(\bx)=0$ defines the so-called {\itshape
bicorn curve} as show in Figure \ref{fig-3} (a). 
Hence, the feasible set $\Omega_3$ of this problem is the
region enclosed by the bicorn curve and the image $f(\Omega_3)$ is
obtained by rotating $\Omega_3$ clockwise by $45^{\circ}$ (Figure
\ref{fig-3} (b)). It is clear that the weakly efficient
solution set $\Sw$ consists of the points on the shorter path connecting the
two singular points of the bicorn curve. As discussed in subsection
\ref{sec::compare}, the linear scalarization \eqref{eq::ls} can only
enable us to compute two points in $\Sw$, namely, the two singular
points of the bicorn curve. 
By our method, we compute the approximation $\mathcal{A}(0.01, 4)$ and
show it in Figure \ref{fig-3} (a), which is the intersection of
$\Omega_3$ and the area under the red curve defined by $\psi_4(\bx)=0.01$. 
The image $f(\mathcal{A}(0.01, 4))$ is illustrated in Figure
\ref{fig-3} (c), which shows that we can obtain good approximations of
$\Sw$ including the ones corresponding to the sunken part of Pareto curve.

Next, we consider the following optimization problem with a Pareto
constraint
\[
	\min\ x_1^2+(x_2-1)^2\quad\text{s.t.}\ (x_1,
	x_2)\in\Sw,
\]
which is to compute the square of the Euclidean distance between the
point $(0,1)$ and the curve $\Sw$. It is easy to see that the unique
minimizer of the above problem is $\left(0, \frac{1}{3}\right)$ and
the minimum is $\frac{4}{9}\approx 0.444$. With the approximation
$\mathcal{A}(0.01, 4)$ of $\Sw$, we consider the polynomial
optimization problem 
\[
	\min\ x_1^2+(x_2-1)^2\quad\text{s.t.}\ \bx\in\Omega_3,\
	\psi_4(\bx)\le 0.01. 
\]
We solve this problem by Lasserre's hierarchy of SDP relaxations (c.f.
\cite{Lasserre2001,Lasserre2009}) with the software GloptiPoly
\cite{gloptipoly}, and get the certified minimizer $(-0.0000, 0.3473)$
and minimum $0.4260$. 
\qed

\begin{figure}
	\centering
%	\subfigure[$\text{The bicorn curve}$]{
%\scalebox{0.3}{
%	\includegraphics[trim=80 200 80 200,clip]{bicorn.pdf}
%}
%}
%	\subfigure[$f(\Omega_3)$]{
%\scalebox{0.3}{
%	\includegraphics[trim=80 200 80 200,clip]{f_bicorn.pdf}
%}
%}\\
	\subfigure[]{
\scalebox{0.4}{
	\includegraphics[trim=80 200 80 200,clip]{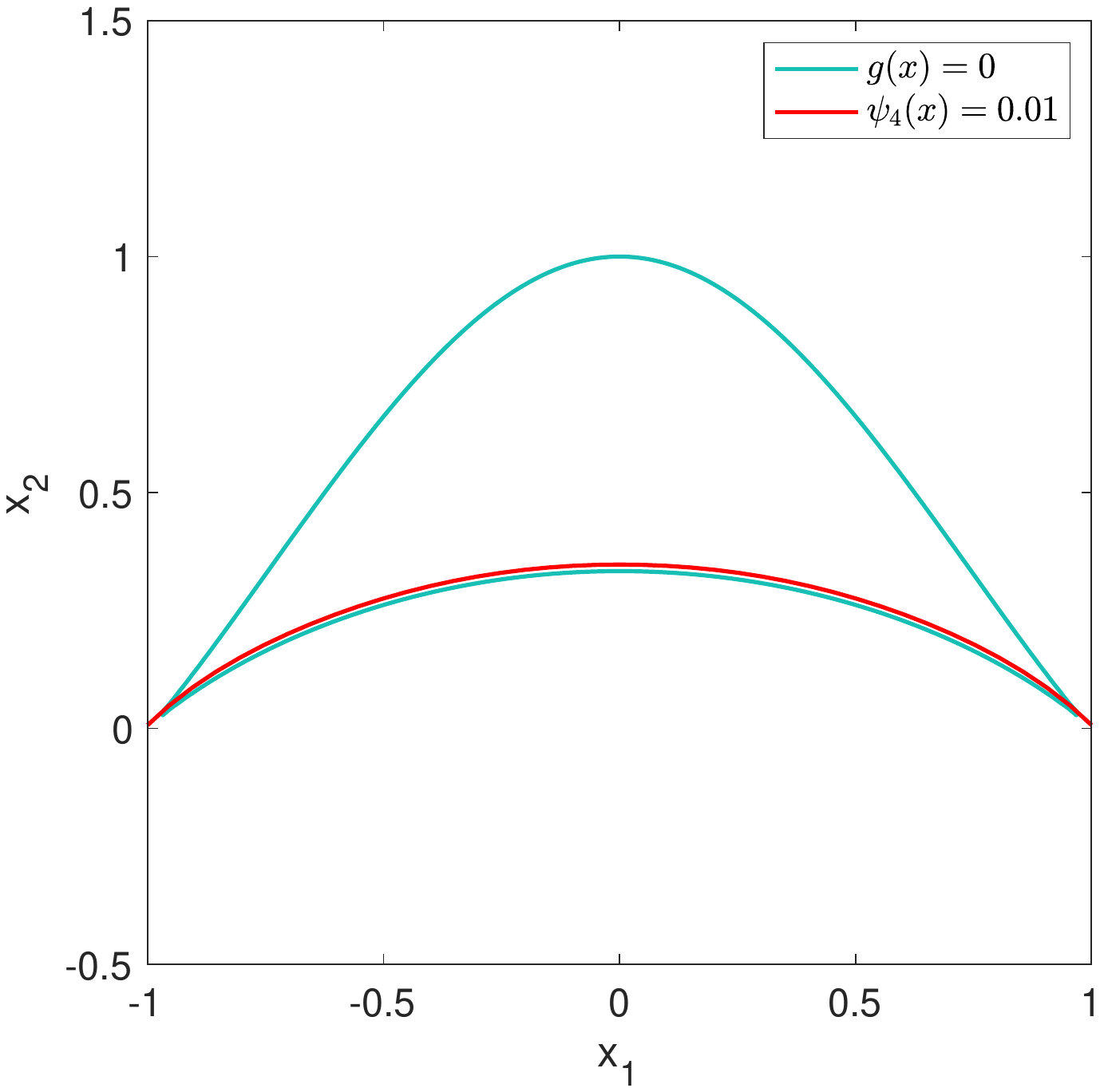}
}
}
	\subfigure[]{
\scalebox{0.4}{
	\includegraphics[trim=80 200 80 200,clip]{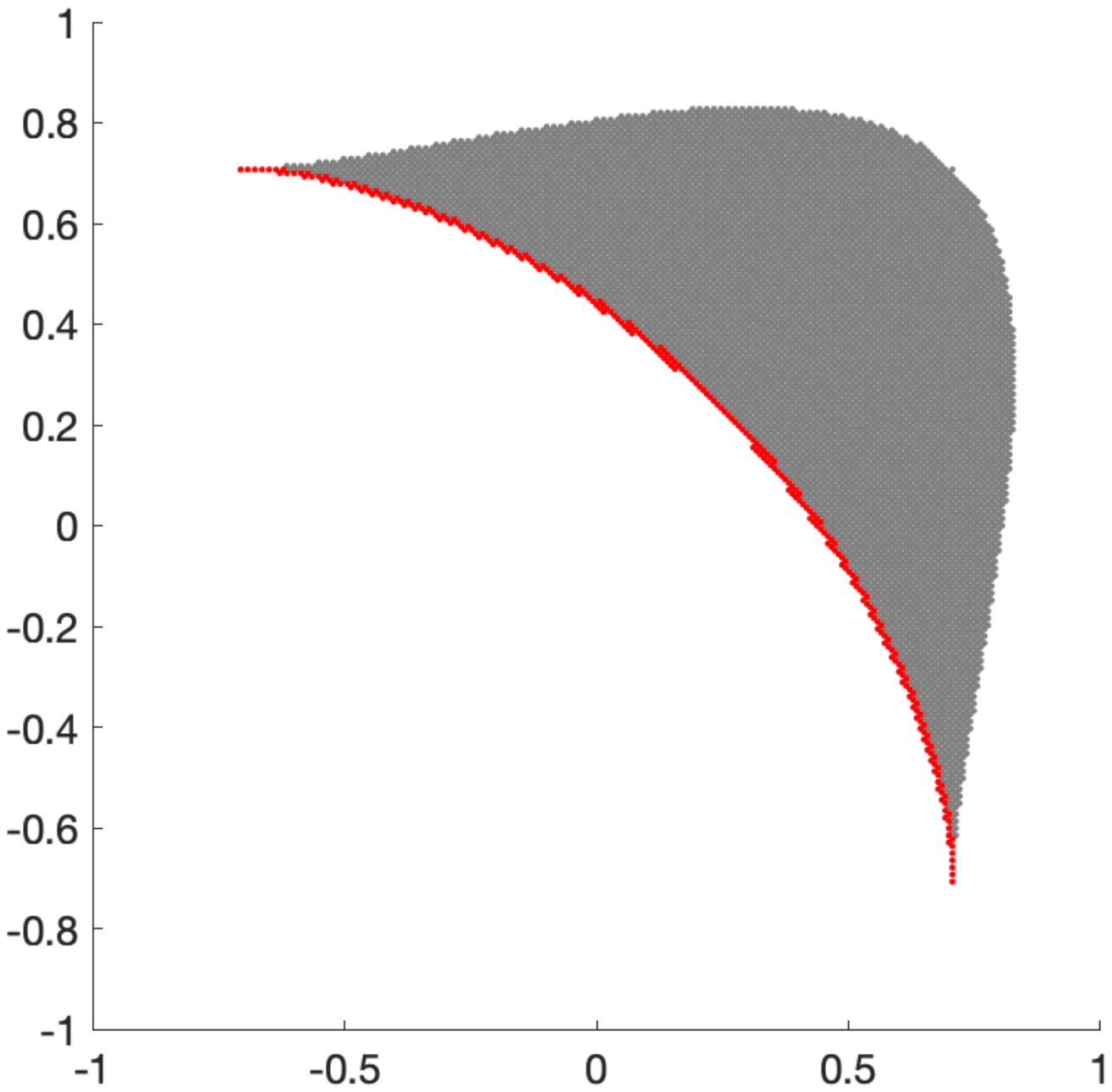}
}
}
\caption{(a) The bicorn curve and the curve defined by
	$\psi_4(\bx)=0.01$; (b) The images $f(\mathcal{A}(0.01, 4))$ (in red) and
$f(\Omega_3)$ (in gray) in Example \ref{ex::3}}
\label{fig-3}
\end{figure}
}
\end{example}

\begin{example}\label{ex::5}{\rm
Consider the problem   
\[	
	\left\{
\begin{aligned}
	{\rm Min}_{\mathbb{R}^2_+} &\ \left(-x_1^2,\ x_1^4+x_2^2\right)\\
	\text{s.t.}\quad &\ \bx\in\Omega_4:=\left\{\bx\in\mathbb{R}^2 \colon 1 - x_1^2 - x_2^2 \ge 0 \right\}.
\end{aligned}
\right.
\]
It is easy to see that the set of weakly efficient solutions $\Sw=[-1,
1]\times \{0\}$ and the image $f(\Sw)$ (the Pareto curve) is the curve
\[
	\left\{(t_1, t_2)\in\mathbb{R}^2 : t_2=t_1^2,\ t_1\in[-1, 0]\right\}
\]
in the objective plane where $t_1=-x_1^2$ and $t_2=x_1^4+x_2^2$.
Clearly, for every point $(t_1, t_2)\in f(\Sw)$, there are two weakly
efficient solutions $(-\sqrt{-t_1}, 0)$ and $(\sqrt{-t_1}, 0)$.
Therefore, this problem does not satisfy the assumptions of the
approach proposed in \cite{Magron2014} when using the scalarizations
\eqref{eq::ls} and \eqref{eq::cs}. 
By our method, we compute the set $\mathcal{A}(0.005, 5)$, which is
the intersection of the unit disk and the area enclosed by the red
curve defined by $\psi_5(\bx)=0.005$ in Figure \ref{fig-5} (a). 
The images $f(\Omega_4)$ and $f(\mathcal{A}(0.005, 5))$ is shown in
Figure \ref{fig-5} (b), which illustrates
that we can approximate the set of weakly efficient solutions as
closely as possible.
\qed

\begin{figure}[h]
	\centering
	\subfigure[]{
\scalebox{0.4}{
	\includegraphics[trim=60 200 80 200,clip]{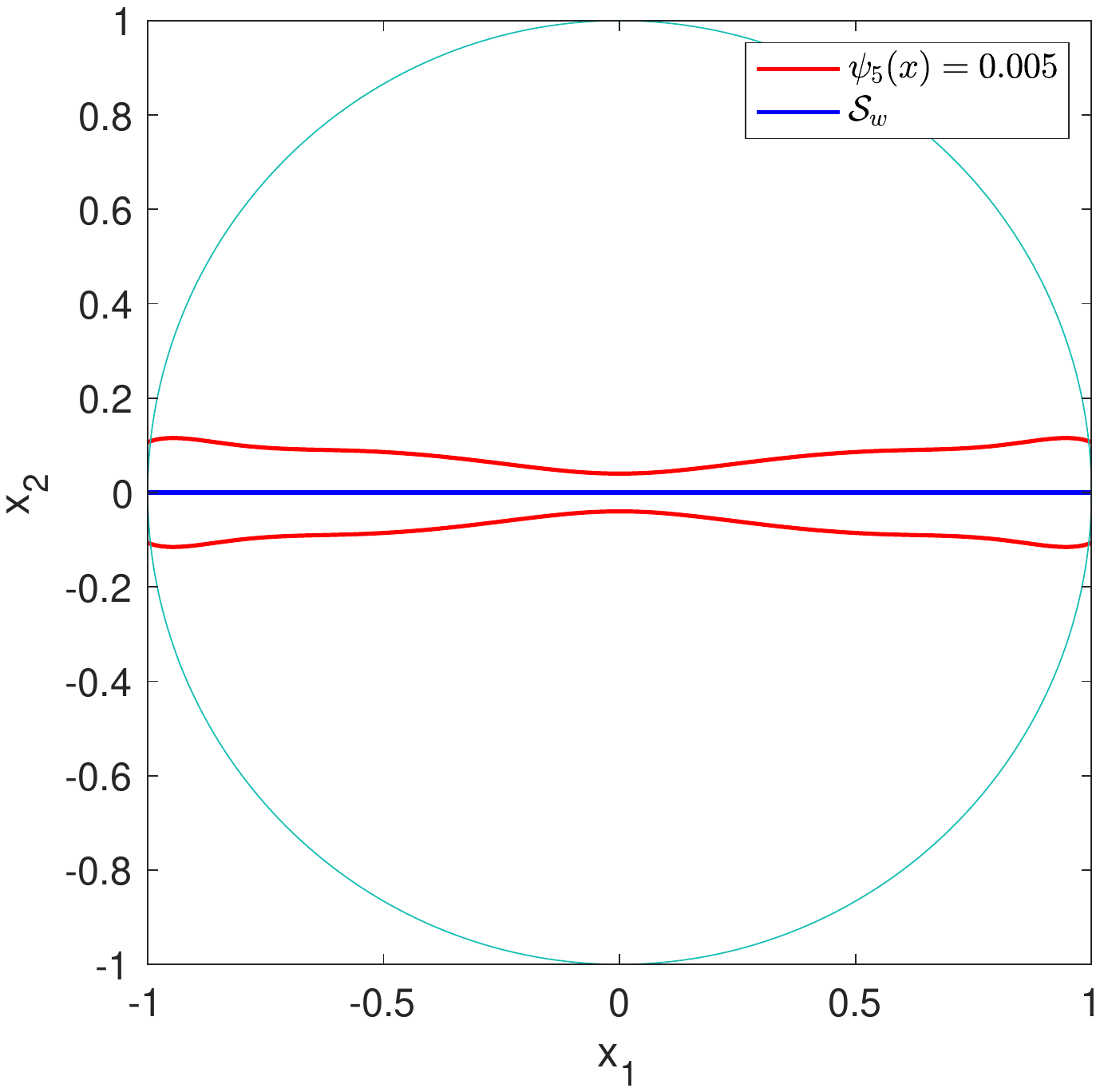}
}
}
	\subfigure[]{
\scalebox{0.4}{
	\includegraphics[trim=60 200 80 200,clip]{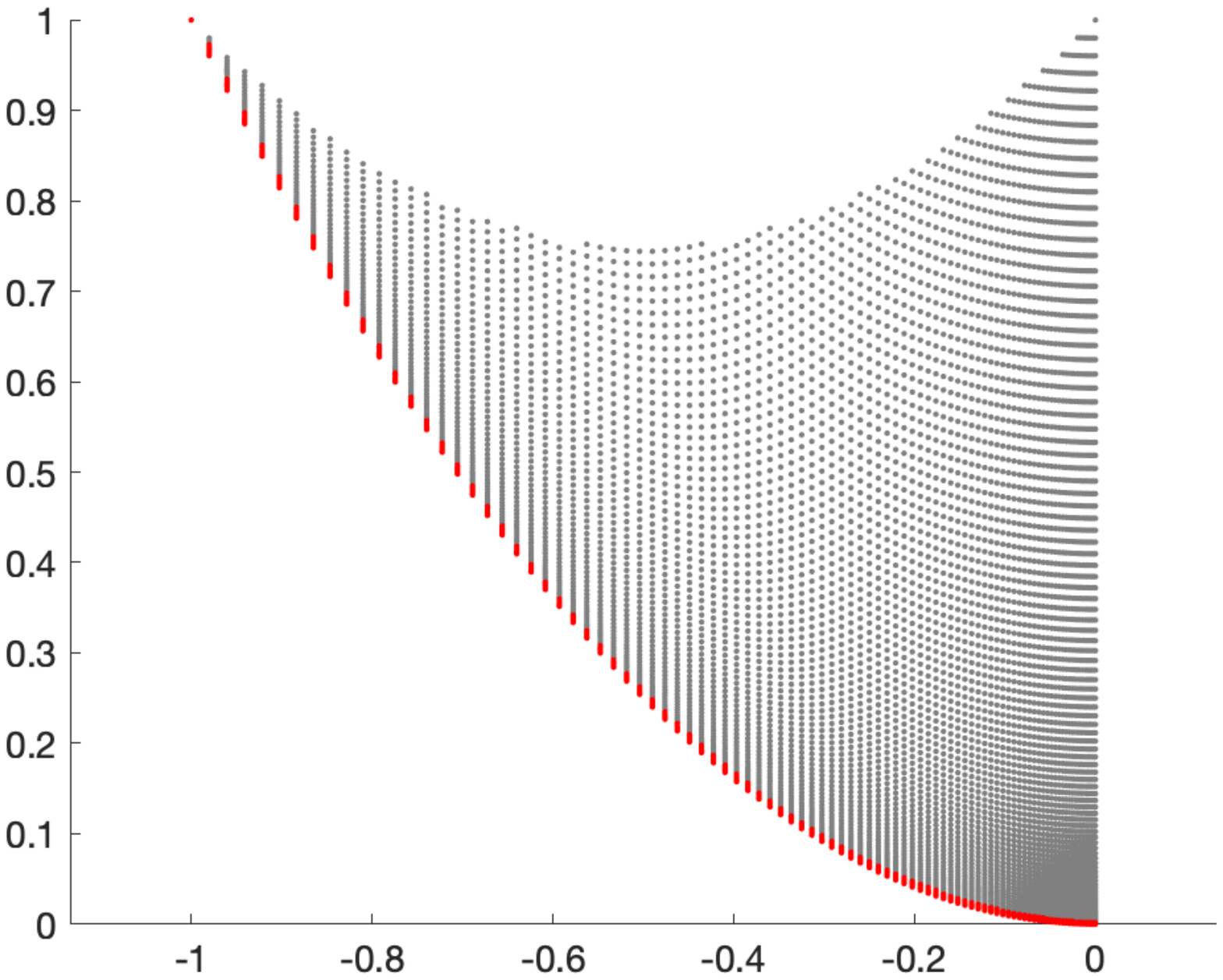}
}
}
%	\subfigure[$f(\mathcal{A}(0.01,5))$]{
%\scalebox{0.3}{
%	\includegraphics[trim=60 200 80 200,clip]{compare_001_5.pdf}
%}
%}
\caption{(a) The set $\Sw$ and the curve defined by
	$\psi_5(\bx)=0.005$; (b) The images $f(\mathcal{A}(0.005, 5))$ (in red) and
$f(\Omega_4)$ (in gray) in Example \ref{ex::5}}.
\label{fig-5}
\end{figure}

}
\end{example}

\section{Conclusions}\label{sec::5}
In this paper, we provide a new scheme for approximating the set of all
weakly ($\be$-)efficient solutions to the problem~\eqref{VP}.
The procedure mainly relies on the achievement function associated
with \eqref{VP} and the ``joint+marginal'' approach proposed by
Lasserre~\cite{Lasserre2015}.
%Nevertheless, our method takes advantage of the existing ones in \cite{Magron2014,Nie2021}, and the 
The obtained results seem new in the area of vector optimization with polynomial structures, 
in the sense that we approximate the whole set of weakly
($\be$-)efficient solutions to the problem~\eqref{VP}.
%not solely find some weakly ($\be$-)efficient solutions.
Moreover, the obtained results also significantly develop the recent
achievements in \cite{Chuong2020,LeeJiao2018,LeeJiao2021} for vector
polynomial optimization problems from convex settings to nonconvex settings.

\subsection*{Acknowledgments}
The authors thank Professor Tien-Son Pham for his helpful comments on the early version of this
manuscript.
Feng Guo was supported by the Chinese National Natural Science Foundation under grant 11571350, and the Fundamental Research Funds for the Central Universities.
Liguo Jiao was supported by Natural Science Foundation of Jilin Province (no. 20220101302JC) and the Fundamental Research Funds for the Central Universities.

%\bibliographystyle{abbrvurl}
%\small
%\bibliography{MPO}
\end{document}